%% file: complexity_polytopal.tex
\documentclass[12pt]{amsart}

\usepackage{amsmath, amssymb, amsthm, mathrsfs}
\usepackage[hidelinks]{hyperref}
\usepackage[margin=1.16in]{geometry} 
\usepackage{comment} % To comment out text.

\usepackage{mathtools}

\usepackage{graphicx}
\usepackage{color}
\usepackage{import} % For some reason my compiler stopped liking \input

\numberwithin{equation}{section}

\title{Cut and project sets with polytopal window I: complexity}
\date{\today}

\author{Henna Koivusalo}
\address{School of Mathematics, University of Bristol, Bristol, BS8 1UG, UK}
\email{henna.koivusalo@bristol.ac.uk}
\urladdr{https://people.maths.bris.ac.uk/~te20281/}

\author{James J.\ Walton} 
\address{School of Mathematical Sciences, Mathematical Sciences Building, University Park, Nottingham, NG7 2RD, United Kingdom}
\email{James.Walton@nottingham.ac.uk}
\urladdr{https://www.nottingham.ac.uk/mathematics/people/james.walton}

\theoremstyle{plain}
\newtheorem{theorem}{Theorem}[section]
\newtheorem{definition}[theorem]{Definition}
\newtheorem{lemma}[theorem]{Lemma}

\newtheorem{proposition}[theorem]{Proposition}
\newtheorem{corollary}[theorem]{Corollary}
\newtheorem*{informal_cut}{Informal version of Corollary \ref{cor: acceptance <-> cuts}}
\newtheorem*{informal_complexity}{Informal version of Theorem \ref{thm: generalised complexity}}

\theoremstyle{definition}
\newtheorem{remark}[theorem]{Remark}
\newtheorem{example}[theorem]{Example}
\newtheorem{notation}[theorem]{Notation}

\newcommand{\id}{\mathrm{id}}
\newcommand{\R}{\mathbb{R}}
\newcommand{\Z}{\mathbb{Z}}

\newcommand{\N}{\mathbb{N}}
\newcommand{\sH}{\mathscr{H}}
\newcommand{\sA}{\mathscr{A}}
\newcommand{\sC}{\mathscr{C}}
\newcommand{\sF}{\mathscr{F}}
\newcommand{\sD}{\mathscr{D}}
\newcommand{\cS}{\mathcal{S}}
\newcommand{\cA}{\mathcal{A}}
\newcommand{\cC}{\mathcal{C}}
\newcommand{\cB}{\mathcal{B}}
\newcommand{\bT}{\mathbb{T}}
\newcommand{\op}{\mathrm{op}}
\newcommand{\rk}{\mathrm{rk}}

\newcommand{\E}{\mathbb{E}}
\newcommand{\tot}{\E}               % total space
\newcommand{\phy}{\tot_{\vee}}      % physical = `real' space
\newcommand{\intl}{\tot_{<}}        % internal space
\newcommand{\cps}{\Lambda}          % cut and project set
\newcommand{\W}{W^c}                % complement of window
\newcommand{\iW}{\mathring{W}}      % interior of window
\newcommand{\iH}{\mathring{H}}

\thanks{Supported by a London Mathematical Society Scheme 2 grant. JW supported by EPSRC grant EP/R013691/1.}

\begin{document}

\thispagestyle{empty}

%===============================================

\begin{abstract}
We calculate the growth rate of the complexity function for polytopal cut and project sets. This generalises work of Julien where the almost canonical condition is assumed. The analysis of polytopal cut and project sets has often relied on being able to replace acceptance domains of patterns by so-called cut regions. Our results correct mistakes in the literature where these two notions are incorrectly identified. One may only relate acceptance domains and cut regions when additional conditions on the cut and project set hold. We find a natural condition, called the quasicanonical condition, guaranteeing this property and demonstrate via counterexample that the almost canonical condition is not sufficient for this. We also discuss the relevance of this condition for the current techniques used to study the algebraic topology of polytopal cut and project sets.
\end{abstract}
\maketitle

\section*{Introduction}

Periodic patterns of Euclidean space have been central geometric objects of study in mathematics for millennia. By the 18\textsuperscript{th} century it was realised by scientists such as Kepler, Steno and Ha\"{u}y that the polyhedral macroscopic geometry of crystalline structures was likely a result of them being periodically arranged at the microscopic level. Based on this assumption, the internal order of such materials could thus be understood through a classification of the space groups, all 230 of which were found by Fedorov and Sch\"{o}nflies by 1892. Some mathematical objects, however, whilst having almost periodic behaviour, lack perfect periodicity. The field of Aperiodic Order is the mathematical study of such structures \cite{BG13}. For example, codings of dynamical systems may have a highly structured and repetitive nature without being perfectly periodic. Properties of linear forms may be understood as properties of quasiperiodic structures. Similarly, materials with highly ordered internal structure need not be periodic, as discovered by Dan Shechtman in 1982 when he first identified what are now known as quasicrystals \cite{SchBleGraCah84}.

The richness of the loosely defined class of aperiodically ordered patterns is immense, although there are still only a few standard ways of generating significant classes of them. One, via substitution rules, generates aperiodically ordered patterns with hierarchical structure. The other major technique of constructing aperiodic patterns is via the cut and project method (the resulting point sets are also known as model sets \cite{Moo97}). With some standard restrictions, the patterns constructed in this way in fact display pure point diffraction, which makes them an important class to study for their relevance in modelling quasicrystals \cite{Hof95b, Hof95a}.

In this paper we restrict attention to cut and project patterns with convex polytopal windows. Loosely speaking, such patterns are constructed as follows. Firstly, we have a Euclidean space $\tot$, called the total space, of dimension $k$. We then choose a $d$-dimensional subspace $\phy \leqslant \tot$ called the physical space and a complementary subspace $\intl$; the notation indicates that we often picture projecting `downwards' to the physical space, and `sideways' to the internal space (see Figure \ref{fig:cps}). A window $W\subset \intl$ is chosen inside the internal space which defines the strip $W + \phy \subseteq \tot$ parallel to the physical space. From this data one constructs cut and project sets by placing a (possibly translated) lattice $\Gamma \leqslant \tot$ in the total space, `cutting' those points of it which fall into the strip and then projecting them to the physical space. Under certain simple restrictions, the resulting point pattern in the $d$-dimensional physical space is aperiodic and has a great deal of structural order.

This paper is the first of two addressing the long-range order of cut and project sets with polytopal windows. In the forthcoming second part, we characterise \emph{linear repetitivity} \cite{LP03} for such patterns, greatly generalising the cubical case solved in \cite{HaynKoivWalt2015a}. The current work concerns the \emph{complexity} of these patterns and builds a framework for analysing the appearance of their finite patterns, which will also be fundamentally utilised in the second part.

One desirable property of an aperiodic pattern is for it to have low complexity, which means that it has a small number of distinct local configurations of a given size, up to translation equivalence. It was established by Julien \cite{Jul10} that, under a certain restriction on the cut and project scheme, the number $p(r)$ of patches of radius $r$ (called $r$-patches) grows, asymptotically, as $r^\alpha$ for aperiodic polytopal cut and project sets, where $\alpha \in \N$ with $d \leq \alpha \leq d(k-d)$. Moreover, the number $\alpha$ can be determined via the ranks of certain intersections of the lattice with subspaces associated to the hyperplanes defining the boundary of the window. Unfortunately, there is an error in the proof of \cite[Proposition 2.1]{Jul10}. One of the purposes of the current work is to explore the severity of this error and correct it. 

The restriction used in \cite{Jul10} is the `almost canonical condition', the purpose of which is to establish a connection between `acceptance domains' of the window and, what we call here, \emph{cut regions}. The acceptance domains are subsets of the internal space which dictate the types of patches seen in the pattern: one lifts points of the cut and project set to the lattice, projects them to the internal space and then the acceptance domain this point lands in determines the local patch at that point. However, the geometry of the acceptance domains is somewhat complicated, and for ease of analysis they are replaced by `cut regions', the convex regions bounded by lattice translates of the hyperplanes defining the boundary of the window.

As we shall see in Example \ref{ex: triangle}, the almost canonical condition is not the correct one to establish the connection of acceptance domains and cut regions. We introduce a new condition, which we call the \emph{quasicanonical} condition, which guarantees the desired connection. This repairs the error in the main argument of \cite{Jul10} in determining the growth of the complexity function for a polytopal cut and project set under the quasicanonical condition. To this end we prove the following result:

\begin{informal_cut} For a quasicanonical cut and project set, up to a linear rescaling of $r$, the acceptance domains $\sA(r)$ of $r$-patches refine cut regions $\sC(r)$, and vice versa.
\end{informal_cut}

We continue the analysis to the situation where the only assumption is that the window is polytopal, proving the following statement as Theorem \ref{thm: generalised complexity}. In the following, $\Gamma^H$ (Definition \ref{def: stabiliser}) is the subgroup of $\Gamma$ preserving a given hyperplane $H$ defining the boundary of the window; the exact statement requires some extra notation which will be postponed to Section \ref{sec: removing the quasicanonical condition}.

\begin{informal_complexity}
Consider an aperiodic cut and project pattern with a convex polytopal window $W$. Then the complexity grows asymptotically as $p(r) \asymp r^\alpha$ for $\alpha \in \N$. The number $\alpha$ can be derived from the ranks of the subgroups $\Gamma^H$ of $\Gamma$ and the dimensions of their $\R$-linear spans in $\intl$.
\end{informal_complexity}

As in \cite{Jul10}, it can be shown that $\alpha \leq d(k-d)$ and that $\alpha \geq d$ if the cut and project set is aperiodic. This result means that the quasicanonical condition need not have been introduced for the question of the complexity of polytopal cut and project sets. However, the quasicanonical condition guarantees the connection between acceptance domains and cut regions, a property that seems important in many applications. In fact, there are many instances in the literature where this connection is needed and to establish it, the almost canonical condition is assumed. For example, this is the case in the topological study of these patterns via their associated translational hulls \cite{FHK02}. Our work demonstrates that in such cases the almost canonical condition should be replaced with the quasicanonical condition.

\subsection*{Outline of the paper}

In Section \ref{sec: euclidean cut and project schemes} we give the definition of a Euclidean cut and project scheme and lay out our notation. In Section \ref{sec: patches and acceptance domains} we recall how one constructs the acceptance domains and how they are associated to patches of cut and project sets. We define the quasicanonical condition in Section \ref{sec: quasicanonical cut and project schemes} for polytopal windows. Geometric intuition is given for what the condition says and also some technical results which will be useful for establishing that under this condition one may connect acceptance domains to cut regions, which is stated precisely and proved in Section \ref{sec: cut regions}. We also contrast the quasicanonical and almost canonical conditions in this section through some simple examples, and demonstrate that in particular cases the connection between acceptance domains and cut regions fails to hold true if only assuming the almost canonical property. We believe that these first sections could be of independent importance in the study of patterns in cut and project sets. 

In Section \ref{sec: complexity} we derive, in a similar manner to Julien's proof in \cite{Jul10}, a formula for the asymptotic rate of growth of the complexity function for a quasicanonical cut and project set. These techniques are refined in Section \ref{sec: removing the quasicanonical condition}, where we obtain a general result on the complexity function by removing the quasicanonical condition. Finally, in Section \ref{sec: topology of transversal}, we explain why the quasicanonical condition is relevant in describing the topology of the translational hull of a polytopal cut and project set and hence also to calculations for their cohomology.

\subsection*{Summary of notation}

\begin{itemize}
		\item $(\tot,\phy,\intl,\Gamma,W)$ : Euclidean cut and project scheme, with total space $\tot$ of dimension $k$, physical space $\phy$ of dimension $d$, complementary internal space $\intl$, lattice $\Gamma \leqslant \tot$ and window $W \subseteq \intl$ (compact and equal to the closure of its interior).
		\item $\pi_\vee \colon \tot \to \phy$ : the projection from the total space to the physical space with respect to $\intl$. For $x \in \tot$ we let $x_\vee \coloneqq \pi_\vee(x)$. Analogous notation for $\intl$ in place of $\phy$.
		\item $y^\wedge$ : for $y = \pi_\vee(\gamma)$, with $\gamma \in \Gamma$, we let $y^\wedge \coloneqq \gamma$, i.e., $y^\wedge \in \Gamma$ and $(y^\wedge)_\vee = y$.
		\item $y^*$ : the star map applied to $y \in \Gamma_\vee$ is $y^* \coloneqq (y^\wedge)_<$.
		\item $\cps$ : the cut and project set associated to the scheme, $\cps \coloneqq \{y \in \Gamma_\vee \mid y^* \in W\}$.
		\item $A_P$ : the acceptance domain in $W$ associated to (pointed) patch $P$ of $\cps$.
		\item $\sA(r)$ : the set of acceptance domains of $r$-patches.
		\item $\sH^+$ : the set of half-spaces defining a polytopal window $W$, with associated collection of affine hyperplanes $\sH$.
		\item $V(H)$ : the vector subspace given by translating an affine hyperplane $H$ over the origin.
		\item $\cS(r)$, $\cB(r)$ : the `slab' and `box' of size $r$ in $\Gamma$ (equations \ref{eq: slab} and \ref{eq: box}). For a polytopal window these define the cut regions $\sC'(r)$ and $\mathscr{C}(r)$.
		\item $\Gamma^H$ : the stabiliser of $H$ is the subgroup $\Gamma^H \coloneqq \{\gamma \in \Gamma \mid H = H + \pi_<(\gamma)\}$ of $\Gamma$.
		\item $f \ll g$ : for two functions $f,g \colon \R_{>0} \to \R_{>0}$, we write $f \ll g$ to mean that there exists some constant $C > 0$ for which $f(x) \leq C g(x)$ for sufficiently large $x$. That is, $f \in O(g)$.
		\item $f \asymp g$ : for two functions as above, $f \asymp g$ means that $f \ll g$ and $g \ll f$.
\end{itemize}

\subsection*{Acknowledgements} Both authors thank Antoine Julien for several helpful discussions in the preparation of this work. He should also be credited with suggesting the name `quasicanonical' for our property. We thank the anonymous referee for their insightful comments and suggestions.

\section{Euclidean cut and project set schemes} \label{sec: euclidean cut and project schemes}

Throughout, vector spaces will be finite-dimensional over $\R$. Let $\tot$ be a $k$-dimensional vector space, which we refer to as the {\bf total space}. We equip it with a decomposition $\tot = \phy + \intl$ where $\dim(\phy) = d$ and $\dim(\intl) = k-d$, with $0<d<k$. The subspaces $\phy$ and $\intl$ are referred to as the {\bf physical space} and the {\bf internal space}, respectively. Let $\Gamma$ be a lattice in $\tot$, a discrete subgroup of the total space which is co-compact, that is, with compact quotient $\tot / \Gamma$. Therefore, $\Gamma \cong \Z^k$ and $\tot / \Gamma \cong (S^1)^k$, the $k$-torus. Finally, we choose a {\bf window} $W \subset \intl$, a compact subset of $\intl$ which is equal to the closure of its interior. In much of what follows we will further restrict to polytopal windows.

Let $\pi_\vee \colon \tot \to \phy$ and $\pi_< \colon \tot \to \intl$ denote the projections from the total space to the physical space and internal space, respectively, where the projections are with respect to the decomposition $\tot = \phy+\intl$. We make the standard assumptions that $\pi_\vee$ is injective on $\Gamma$ and that $\pi_<(\Gamma)$ is dense in $\intl$.

The data $(\tot,\phy,\intl,\Gamma,W)$ is called a {\bf $k$-to-$d$ Euclidean cut and project scheme}.

\begin{notation} Throughout, given $X \subseteq \tot$, we let $X_\vee \coloneqq \pi_\vee(X)$ and $X_< \coloneqq \pi_<(X)$. Similarly, for $x \in \tot$ we write $x_\vee \coloneqq \pi_\vee(x)$ and $x_< \coloneqq \pi_<(x)$.

In the literature there are two standard pictures of geometrically representing cut and project schemes. In some contexts, such as where a special kind of window is given, it is helpful to consider what happens when $\phy$ is being varied; that is, $\Gamma$ is considered fixed, often with $\tot = \R^k$ and $\Gamma = \Z^k$, and $\phy$ and $\intl$ are `skewed' in $\tot$. This is the approach of \cite{HKW14, HaynKoivSaduWalt2015, HaynKoivWalt2015a}, for example. In other contexts it is preferable to think of $\phy$ and $\intl$ as fixed with $\Gamma$ varying, taking $\tot \coloneqq \phy \oplus \intl$. This is the case in \cite{MS15}. Of course these viewpoints are ultimately equivalent and since we only specify that $\tot$ has the direct sum decomposition $\phy+\intl$, rather than being defined as the direct sum, our conventions are agnostic as to which is preferred.

Our notation, however, is chosen with the second viewpoint in mind. For a $2$-to-$1$ scheme, we imagine $\phy$ as the $x$-axis and $\intl$ as the $y$-axis of $\tot = \R^2$, see Figure \ref{fig:cps}. The notation indicates that $\phy$ is the subspace to which one projects `downwards', and $\intl$ is the subspace to which one projects `to the left' (at least, of course, restricting view to points in the first quadrant, as in the standard pictures). Similarly, for $x \in \Lambda$, the point $x_\vee$ may be thought of as the projection of the point `downwards' to $\phy$, and $x_<$ is the projection of $x$ `to the left'.
\end{notation}

Since $\pi_\vee$ is injective on $\Gamma$, for every $y \in \Gamma_\vee$ there is a unique element, denoted $y^\wedge \in \Gamma$ and called the {\bf lift of $y$}, for which $(y^\wedge)_\vee = y$. Again, this notation is geometrically inspired by Figure \ref{fig:cps}. An important map in analysing the nature of recurrence of patches in cut and project sets is the {\bf star map} $y \mapsto y^*$ mapping $\Gamma_\vee$ into $\intl$, and defined by $y^* \coloneqq (y^\wedge)_<$. Note that the star map is the composition of the inverse of the isomorphism $\pi_\vee |_\Gamma \colon \Gamma \xrightarrow{\cong} \Gamma_\vee$ followed by the projection $\pi_<$, so the star map is a homomorphism (although it is not continuous in the subspace topology of $\Gamma_\vee \subseteq \phy$).

We call a parameter $s \in \tot$ {\bf singular} if $\pi_<(\Gamma+s) \cap \partial W \neq \emptyset$, otherwise $s$ is {\bf non-singular}. The {\bf cut and project set} associated to the cut and project scheme and non-singular parameter $s$ is the subset of the physical space defined by:
\begin{equation} \label{eq: cpsfull}
\cps_s \coloneqq \{\gamma_\vee \mid \gamma \in (\Gamma+s) \cap (\phy+W)\}.
\end{equation}
That is, we form the cut and project set by `cutting' out the points of the shifted lattice $\Gamma+s$ which fall in the `strip' $\phy+W$, and then project them to the physical space; see Figure \ref{fig:cps}. 

\begin{figure}
	\def\svgwidth{\textwidth}
	\centering
	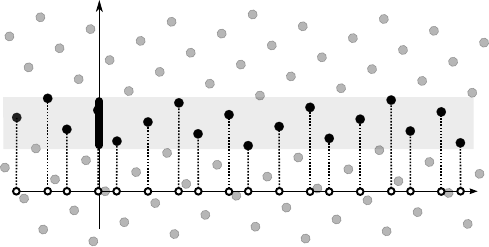
	\caption{The definition of a cut and project set.}\label{fig:cps}
\end{figure}

There are also cut and project sets defined by singular parameters, although multiple ones for each such parameter. In this case, to get permissible cut and project sets one should make choices---which are consistent, in a certain sense---as to which points $\gamma_\vee$ with $\gamma_< \in \partial W$ are included. These choices should be made so that the resulting finite sub-patches appear in the $\Lambda_s$ coming from non-singular parameters, which is to say that these singular patterns are certain limits of non-singular $\Lambda_s$.

It follows from density of $\Gamma_<$ in $\intl$ that every $\cps_s$ for non-singular $s$ has the same set of `finite patches'. As such, the parameter $s$ does not play an important role for questions concerning complexity. Hence, shifting the window $W$ if necessary, we may assume without loss of generality that our chosen parameter is $s=0$ and is non-singular. Correspondingly, we denote $\cps \coloneqq \cps_0$, and observe that Equation \ref{eq: cpsfull} becomes:
\begin{equation} \label{eq: cps}
\cps = \{y \in \Gamma_\vee \mid y^* \in W\}.
\end{equation}
A {\bf period} of $\cps$ is some $x \in \phy$ with $\cps = \cps + x$. We call $\cps$ {\bf aperiodic} if it has no non-zero periods. It is not hard to show that the set of periods corresponds to the kernel of the star map, so $\cps$ is aperiodic if and only if $\pi_<$ is injective on $\Gamma$. We shall not assume aperiodicity unless otherwise stated.

%%%%%%%%%%%%%%%%%%%%%%%%%%%%%%%%%
%%%%%%%%%%%%%%%%%%%%%%%%%%%%%%%%%

\section{Patches and acceptance domains} \label{sec: patches and acceptance domains}

Some of the content of this section is very closely related to earlier works, see \cite{Jul10, BV00, HaynKoivSaduWalt2015}. 

By identifying $\tot \cong \R^k$ we get an inner product, norm and metric on the total space, and similarly on the physical and internal space. The particular choice does not affect our proofs.
Denote a closed ball of radius $r$ and centre $x$ by $B_r(x)$. For $y \in \cps$ and $r>0$ we define $P(y,r)$, the {\bf $r$-patch} of $y \in \cps$, to be $\cps \cap B_r(y)$, the subset $\cps$ of points within distance $r$ of $y$, together with the specially marked `origin' $y \in P(y,r)$. We say that the $r$-patches at $x$ and $y$ are {\bf translation equivalent} if $P(x,r) - x = P(y,r) - y$, that is they agree up to a translation mapping the marked origin of one to the other.

Cut and project sets, as defined above, have the special property of being {\bf repetitive}, that is, for every $r > 0$ there exists some $R_r$ for which every $r$-patch, up to translation equivalence, can be found within $R_r$ of every point of $\phy$. Let $p(r)$ denote the number of translation classes of $r$-patches in some cut and project set $\cps$. This is called the {\bf complexity function} of $\cps$.

The translation class of an $r$-patch at $y \in \cps$ can be determined by which `acceptance domain' inside the window $W$ the point $y^*$ falls into, and hence will play an important role in analysing the complexity function. We explain here how to construct these acceptance domains.

\begin{lemma} \label{lem: displacements} Let $y \in \cps$. If $y+x \in \cps$ (`the displacement $x$ occurs in $\cps$ at $y$') then $x \in \Gamma_\vee$. Furthermore, given $x \in \Gamma_\vee$, then $y+x \in \cps$ if and only if $y^* \in W-x^*$. \end{lemma}

\begin{proof} We have that $y$, $y+x \in \cps$, so they lift to $y^\wedge$, $(y+x)^\wedge \in \Gamma$. Hence $(y+x)^\wedge - y^\wedge \in \Gamma$ and projects under $\pi_\vee$ to $x$, so $x \in \Gamma_\vee$.

Now assume that $y \in \cps$ and $x \in \Gamma_\vee$ are given. By the definition of the cut and project set (Equation \ref{eq: cps}), $y+x \in \cps$ if and only if $(y+x)^* \in W$. Since the star map is a homomorphism, this in turn is equivalent to $y^* \in W-x^*$. \end{proof}

Consider the `slab' of lattice points of width $r$ and $\E_<$ coordinate in $W-W$:
\begin{equation} \label{eq: slab}
\cS(r) \coloneqq \{\gamma \in \Gamma \mid \|\gamma_\vee\| \leq r, \gamma_< \in W-W\}.
\end{equation}
Taking $B=B_r(0)\subseteq \phy$ and $W-W \subseteq \intl$ as also subsets of $\tot$, we may write $\cS(r) = \Gamma \cap (B + (W-W))$. Since $B$ and $W-W$ are bounded, and $\Gamma$ is a lattice, $\cS(r)$ is a finite set for each $r$.

Let $y \in \cps$, so that $y^* \in W$. By the above, it is only possible that $y+x \in \cps$ if $y^* \in W - x^*$, which implies that $x^* \in W-y^* \subseteq W-W$. In specifying the $r$-patch at $y$, one needs to specify those displacements $x \in {\phy}$ of magnitude at most $r$ which are \emph{in} the patch, so that $y+x \in \cps$, and those displacements which are \emph{not in} the patch, so that $y+x \notin \cps$. For each collection one only needs to consider those \emph{possible} displacements, so that $x \in \Gamma_\vee$, $\|x\| \leq r$ and $x^* \in W-W$. So given $P = P(y,r)$, let
\begin{align} \label{eq: in and out}
P_\text{in} \coloneqq \{\gamma \in \cS(r) \mid y+\gamma_\vee \in \cps\}, \\
P_\text{out} \coloneqq \{\gamma \in \cS(r) \mid y+\gamma_\vee \notin \cps\}.
\end{align}

We denote the complement of $W$ in $\intl$ by $W^c$ and its interior by $\iW$.

\begin{corollary}\label{cor: acceptance domains} Let an $r$-patch $P=P(y,r)$ be given. With the definitions of $P_\text{in}$ and $P_\text{out}$ above, we have that the $r$-patch at $z \in \cps$ agrees with $P$ up to translation if and only if $z^* \in A_P$, where
	\begin{equation} \label{eq: acceptance domain}
A_P = \left(\bigcap_{\gamma \in P_\text{in}} \iW - \gamma_<\right) \cap \left(\bigcap_{\gamma \in P_\text{out}} W^c - \gamma_<\right).
	\end{equation}
\end{corollary}

\begin{proof} The $r$-patches $P(y,r)$ and $P(z,r)$ agree if and only if the sets of displacements
\[
\{y'-y \mid y' \in P(y,r)\} = \{z'-z \mid z' \in P(z,r)\}
\]
agree which, by injectivity of $\pi_\vee$ on $\Gamma$ and the discussion above, is if and only if
\[
P(y,r)_{\mathrm{in}} = P(z,r)_{\mathrm{in}} \ \ \mathrm{and} \ \ P(y,r)_{\mathrm{out}} = P(z,r)_{\mathrm{out}}.
\]
The result then follows from Lemma \ref{lem: displacements}, at least after replacing each $\iW - \gamma_<$ with $W - \gamma_<$. But these sets only differ on their boundaries $\partial W - \gamma_<$, which are singular points, so either of $y^*$ and $z^*$ belong to one if and only if they belong to the other. \end{proof}

\begin{definition} The $A_P \subseteq W$, defined as in the above corollary as certain intersections of $\Gamma_<$ translates of $\iW$ and $\W$, are called {\bf acceptance domains}. We let $\sA(r)$ denote the set of acceptance domains $A_P$ of $r$-patches. \end{definition}

Note that we take the convention here that acceptance domains are open. One could equally take closed domains, for example by replacing $\iW$ with $W$ and $W^c$ with $(\iW)^c$. Doing this can result in translates of $(\iW)^c$ intersecting along lower dimensional subspaces which one would rather ignore, so we prefer here to take open domains (which is also marginally more efficient for defining the cut domains later). These distinctions are not important, since they only differ at points of $\partial W + \gamma_<$ for $\gamma \in \Gamma$, which are singular and thus not relevant in the above corollary.

The window is tiled by the acceptance domains of $\sA(r)$, in the sense that distinct acceptance domains are disjoint and $W = \bigcup_{\sA(r)} \mathrm{cl}(A_P)$, using the fact that the window is the closure of its interior and denseness of $\Gamma_<$ in $\intl$.

\section{Quasicanonical cut and project schemes} \label{sec: quasicanonical cut and project schemes}

To $\eta$, $\alpha \in \intl$ with $\eta \neq 0$ we can associate an affine hyperplane and a half-space:
\begin{align*}
H \coloneqq \{x + \alpha \in \intl \mid x \in \intl, x\cdot \eta = 0\}, \\
H^+ \coloneqq \{x + \alpha \in \intl \mid x \in \intl, x\cdot \eta \geq 0\}.
\end{align*}
In the following, we will usually use the name `hyperplane' for affine hyperplanes, using the term `codimension $1$ subspace' when it is needed that the hyperplane contains the origin.

The window $W$ is called {\bf polytopal} if it is compact, has non-empty interior and may be written as a finite intersection of half-spaces in $\intl$. Let $\sH^+$ denote the set of half-spaces defining $W$. It is uniquely determined by specifying that it is {\bf irredundant}, that is, omitting any element results in a strictly larger intersection. Each $H^+ \in \sH^+$ has an associated affine hyperplane $H$, and we denote the set of these by $\sH$. Each element of $\sH$ (as $\sH^+$ is irredundant) intersects $W$ in a $(k-d-1)$-dimensional `face'; for more details, see Section \ref{sec: complexity}.

We now introduce a condition on polytopal cut and project schemes which allows for a simpler analysis of the corresponding acceptance domains, as we shall see later.

\begin{definition} \label{def: qc} We call a cut and project scheme {\bf quasicanonical} if it is polytopal and we have the following. For each half-space $H^+ \in \sH^+$, with opposite half-space $H^- = (H^+)^c \cup H$ and associated affine hyperplane $H \in \sH$, and each $z \in H$, there exists some $\epsilon > 0$ and a finite collection of points $x_1$, \ldots, $x_m \in \Gamma_<$ so that one of the two equations below is satisfied:
\begin{align} \label{eq: qc1}
H^+ \cap B_\epsilon(z) = \left(\bigcup_{i=1}^m (W - x_i) \right) \cap B_\epsilon(z); \\
\label{eq: qc2}
H^- \cap B_\epsilon(z) = \left(\bigcup_{i=1}^m (W - x_i) \right) \cap B_\epsilon(z).
\end{align}
\end{definition}

Note that even though $\Gamma_<$ is dense, the only translates that are relevant in Equations \ref{eq: qc1} and \ref{eq: qc2} are those that translate a face of $W$ parallel to $H$ into $H$ (see Lemma \ref{lem: qc => face covering}). This is because in Definition \ref{def: qc} we only allow finite unions. The first formula says that the half-space $H^+$ corresponding to one face of the window can be exactly covered, in some neighbourhood of $z$, by a finite collection of $\Gamma_<$ translates of the window. The second formula says the same thing for the opposite half-space. It may be possible that a mixture of the two possibilities is needed over all points of $H$, see Figure \ref{fig:quasicanonical}. A set of translates like those in the rightmost picture of Figure \ref{fig:quasicanonical} are not sufficient to establish the condition for that hyperplane.

\begin{figure}
	\def\svgwidth{0.6\textwidth}
	\centering
	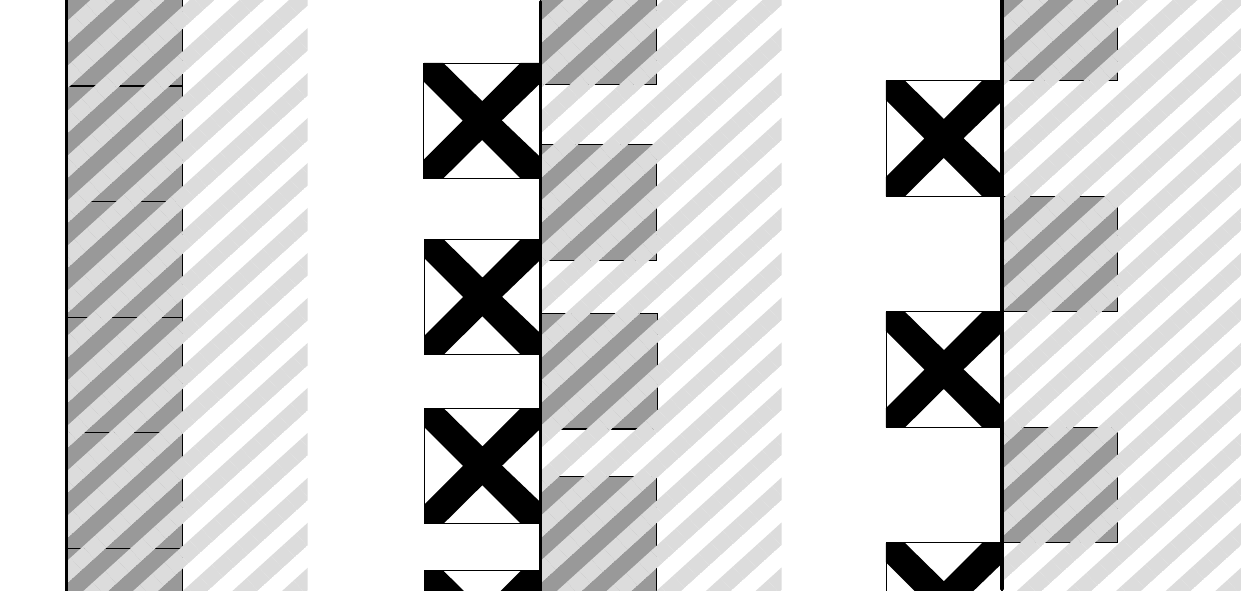
	\caption{In this picture, the grey translates of the window are used in satisfying Equation \ref{eq: qc1} at appropriate points of the supporting hyperplane $H$, given here by a vertical line. The translates marked with an `X' are used in satisfying Equation \ref{eq: qc2} or, equivalently, are those used to exclude points of the opposite half space, as in Remark \ref{rem: duality}. In the leftmost picture, the quasicanonical condition is satisfied: at all points of the $1$d hyperplane, translates of the square window can be used to locally cover the greyscale half-space. The middle picture also satisfies the quasicanonical condition, but at some points of the hyperplane one needs to cover this half-space, and at other points the opposite one (thus using both Equation \ref{eq: qc1} and \ref{eq: qc2}). The rightmost picture does not satisfy the quasicanonical condition: at the `corners', no translates of the window cover the half-space, or its opposite, in any neighbourhood of that point.
}\label{fig:quasicanonical}
\end{figure}

Notice that the quasicanonical condition, in particular, ensures that the affine hyperplanes of $\sH$ may be covered by $\Gamma_<$ translates of $\partial W$ (see Lemma \ref{lem: qc => face covering}). In Section \ref{sec: troublesome scheming} we will compare the quasicanonical condition to the almost canonical condition introduced in \cite{Jul10}, and find out that neither is stronger or weaker than the other.

\begin{remark} \label{rem: duality} Applying complements to Equation \ref{eq: qc2} in $B_\epsilon(z)$, we have the following `dual' form:
\begin{equation} \label{eq: qc opp}
(H^+ \setminus H) \cap B_\epsilon(z) = \left(\bigcap_{i=1}^m W^c - x_i \right) \cap B_\epsilon(z).
\end{equation}
By Lemma \ref{lem: qc => face covering} below, every point of $H$ is singular. So we see that we have equality
\[
H^+ \cap B_\epsilon(z) = \left(\bigcap_{i=1}^m \W - x_i \right) \cap B_\epsilon(z)
\]
provided that we restrict the above formula to the non-singular points of the internal space. This means that, loosely speaking, Definition \ref{def: qc} can be rephrased as saying that, for each $H \in \sH$, at each point of $z \in H$, we can either cover the half-space $H^+$ locally with $\Gamma_<$ translates of $W$, or express this half-space by excluding such translates on the opposite side of the half-space. \end{remark}

\subsection{Boolean Schemes} \label{sec: Boolean} Suppose that for each $z \in H$ we may find translates satisfying the above condition of Equation \ref{eq: qc opp} (or equivalently Equation \ref{eq: qc2}, so that we do not need Equation \ref{eq: qc1}). Then the window satisfies the following {\bf Boolean} condition: we may find a finite number of elements $\gamma_i \in \Gamma$ so that
\[
W^c \cap (W+(W-W)) = \left(\bigcup_{i=1}^\ell W-(\gamma_i)_<\right) \cap (W+(W-W)),
\]
at least after ignoring the singular points of $\partial W$. The set $W+(W-W)$ here should just be thought of as a neighbourhood of $W$: the condition essentially says that we may express the complement of $W$, within a certain neighbourhood, using a union of $\Gamma_<$ translates of $W$.

It follows that we may replace each term $W^c - \gamma_<$ in Equation \ref{eq: acceptance domain} (for $\gamma \in P_\mathrm{out}$) by $\bigcup W - (\gamma-\gamma_i)_<$, since points outside of $W+(W-W)$ are not relevant for the acceptance domains, which necessarily lie in $W$.

Let $c$ denote the maximum of the norms of the $(\gamma_i)_\vee$, so $\|(\gamma - \gamma_i)_\vee\| \leq r+c$. For $y \in \cps$, according to Lemma \ref{lem: displacements}, we have that $y^* \in W^c - \gamma_<$ if and only if $y + \gamma_\vee \notin \cps$, and that $y^* \in \bigcup (W-(\gamma-\gamma_i)_<)$ if and only if one of $y + (\gamma - \gamma_i)_\vee \in \cps$. Therefore, in the Boolean case, to determine if a displacement within radius $r$ \emph{is not} in $\cps$, it is equivalent to show that at least one of a particular finite set of displacements within radius $r+c$ \emph{is} in $\cps$.

The Boolean schemes are special cases of quasicanonical ones, but the Boolean condition is strictly stronger, since it is necessary that if $H^+ \in \sH^+$ then a translate of its opposite half-space is also in $\sH^+$. For example, one may easily define a triangular window which is quasicanonical but not Boolean.

\begin{example} \label{exp: canonical} The data $(\tot,\phy,\intl,\Gamma,W)$ is a {\bf canonical cut and project scheme} if the window is given as the projection of (a translate of) a fundamental cell for $\Gamma$, that is $W = \widetilde{W}_<$ for
\[
\widetilde{W} \coloneqq \{\sum_{i=1}^k \lambda_i b_i \in \tot \mid \lambda_i \in [0,1]\},
\]
where $\{b_i\}_{i=1}^k$ is a basis for $\Gamma$. For example, when $\Gamma = \Z^k$ in total space $\tot = \R^k$ with standard basis, the window is the projection of the unit hypercube $[0,1]^k \subseteq \R^k$ to the internal space.

As one would hope, a canonical scheme is quasicanonical. In fact, it can be shown that it has the Boolean condition above.
\end{example}

\subsection{Informal discussion: Hoopla} 

\begin{figure}
	\def\svgwidth{\textwidth}
	\centering
	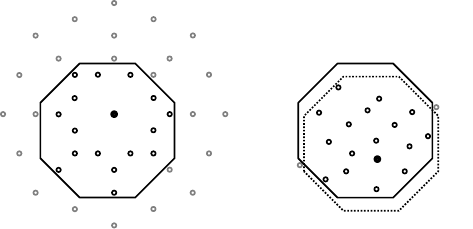
	\caption{On the left: Points inside the translated window belong to the $r$-patch. On the right: The two translates correspond to the same $r$-patch, but external points block a continuous movement of one to the other: some such points have to be temporarily added during the motion. The acceptance domain is disconnected.}\label{fig:hoopla}
\end{figure}

Rather than varying the points $y^*$ in $W$ and considering which $r$-patches arise depending on where precisely $y^* \in W$ sits, one may prefer to imagine a dual picture: translate $y^*$ to the origin, and so consider only $r$-patches at the origin by moving $W$ instead. We play a game of \emph{hoopla}, a carnival ring tossing game: throw $\partial W$ (carefully: by translation) over the finite set of `pegs' $\cS_<(r) \coloneqq \pi_<(\cS(r))$. You win a prize if you capture the special peg $0 \in \cS_<(r)$. The prize is an $r$-patch $P$, and which $r$-patch it is depends only on which pegs land inside the window.

There is a certain amount of room in which to shift $\partial W$ so as to not bump into any pegs and thus still keeping $P$. The full set of motions in which we may move the window in this way traces out a connected region which is contained in the acceptance domain $A_P$ of the patch. However, it is vital to note that it is \emph{not} necessarily the whole of $A_P$, it is just one of possibly several connected components of $A_P$. The window $W$ is convex, so there is no issue in moving it between configurations resulting in the same $P$ so long as one removes pegs outside of the placed window. However, one must also consider these external pegs. It can happen that they obstruct the motion of the window between two configurations yielding $P$, making $A_P$ non-convex or even disconnected.

In the Boolean case of Subsection \ref{sec: Boolean}, this effect is not so important: we can at least manoeuvre the window between two configurations representing the same $r$-patch without bumping into pegs if they correspond to the same $(r+c)$-patch for some fixed $c > 0$. We shall also see that in the quasicanonical case the situation is tractable in a similar way.

\subsection{Preliminary consequences of the quasicanonical condition}

Here we collect some useful technical results on quasicanonical schemes.

\begin{lemma} \label{lem: qc => face covering} One may assume that the translates $W-x_i$ in Equation \ref{eq: qc1} are such that $z \in F-x_i \subseteq H$, where $F = W \cap H$ is the face in $H$. Similarly, in Equation \ref{eq: qc2}, one may assume that the translates $W^c-x_i$ are such that $z \in F^\op-x_i \subseteq H$, where $F^\op$ is the face opposite to $F$, if it exists, given by $F^\op = W \cap H^\op$ where $H$ and $H^\op$ are parallel and $H^\op \in \sH$. \end{lemma}

\begin{proof} Since $F \subseteq H$, clearly if $z \in F-x_i$ then $F-x_i \subseteq H$. We shall show that if $z \notin F-x_i$, then we may discard the term $W-x_i$ in Equation \ref{eq: qc1} (perhaps after reducing the size of $\epsilon$).

Firstly, note that we may as well assume that $z \in W-x_i$. Indeed, otherwise, a sufficiently small ball about $z$ does not intersect this translate, so we may omit it from the union.

Given that $z \in W-x_i$, we have that $z \in F-x_i$ if and only if $z \in H-x_i$. If $z \notin H^+ - x_i$ then $z \notin W-x_i$, so this translate has already been discarded. Suppose, on the other hand, that $z$ belongs to $(H^+ \setminus H) - x_i$. Since $z \in H$, this is equivalent to $x_i \in H^+ \setminus H$, so $-x_i$ is in the complement of $H^+$. Choose any point $f \in F$. By convexity of $W$, the line segment between $f-x_i$ and $z$ lies wholly in $W-x_i$. But the points of this line segment are also in the complement of $H^+$, except for $z$ itself. Indeed, since $f \in H$ and $-x_i \notin H^+$, we see that $f-x_i \notin H^+$, so the same is true of all points between $f-x_i$ and $z \in H$. Hence there are arbitrarily close points to $z$ which are in $W-x_i$ but also in the complement of $H^+$, so Equation \ref{eq: qc1} cannot hold, a contradiction. So we may assume that $z \in H-x_i$, as required.

The case of Equation \ref{eq: qc2} is analogous. Note that if there is no face opposite to $F$ then Equation \ref{eq: qc2} is not used for any $z$ in establishing that the scheme is quasicanonical. \end{proof}

\begin{notation} Given an affine hyperplane $H$ we let $V(H)$ denote its corresponding parallel vector subspace. That is, $V(H) \coloneqq H-H = H-z$ for any $z \in H$. \end{notation}

\begin{lemma}\label{lem: qc => lattice} For a quasicanonical scheme, for each $H \in \sH$, we have that $\Gamma_< \cap V(H)$ contains a basis for $V(H)$. \end{lemma}

\begin{proof} Let $H(A)$ be the set of $z \in H$ for which we may use Equation \ref{eq: qc1} in the quasicanonical definition, and $H(B)$ be the positions where we may use Equation \ref{eq: qc2}; so $H \subseteq H(A) \cup H(B)$. 

Given $z \in H(A)$, choose $a(z) \in \Gamma_<$ so that $z \in F-a(z)$, where $F$ is the face of $W \cap H$. By the above lemma, we may assume that each $x_i$ of Equation \ref{eq: qc1} satisfies this, so there is such a point. This implies that $a(z) \in F-z \subseteq V(H)$, and is at most the diameter $\delta$ of $W$ in distance from $f-z$, for some fixed $f \in F$.

Similarly, given $z \in H(B)$, choose $b(z) \in \Gamma_<$ so that $z \in F^\op-b(z)$, where $F^\op$ is the face opposite to $F$ (if this face doesn't exist then neither does the $z \in H(B)$). Fixing some $z_0 \in H(B)$, we have that each $b(z) - b(z_0) \in (F-z) - (F-z_0) \subseteq V(H)$, and $b(z) - b(z_0)$ is at most $2\delta$ from $z_0-z$.

In summary, every point of $V(H)$ is a bounded distance from a point $a(z)$ or $b(z) \in \Gamma_<$, so $\Gamma_< \cap V(H)$ must contain a basis for $V(H)$, as required. \end{proof}

\begin{lemma} In the definition of quasicanonical one may assume that the same $\epsilon > 0$ is taken for each $z \in H$. \end{lemma}

\begin{proof} By the lemma above there exists a lattice $\mathscr{L} \leqslant \Gamma_< \cap V(H)$. There is some $R>0$ for which $H$ is covered by $\mathscr{L} + B_R(h)$ for any given $h \in H$. Let $\epsilon(z)$ be the $\epsilon$ required for $z$ in Equation \ref{eq: qc1} or \ref{eq: qc2}. By compactness, there is a finite collection $\mathscr{B}$ of balls of the form $B_{\epsilon(z)}(z)$ which covers $B_R(h) \cap H$ for a now fixed $h \in H$. Therefore, by periodicity and compactness, the locally finite covering $\mathscr{B} + \mathscr{L}$ of balls of the form $B_{\epsilon(z)}(z) + \ell$ for $\ell \in \mathscr{L}$ has a Lebesgue number, some $\lambda > 0$ so that every $z \in H$ is such that $B_\lambda(z)$ is contained in a ball of $\mathscr{B} + \mathscr{L}$. The intersections of Equations \ref{eq: qc1} and \ref{eq: qc2} for some $B_{\epsilon(z)}(z)$ also given a corresponding valid intersection for $B_{\epsilon(z)}(z+\ell)$ for $\ell \in \mathscr{L}$, by replacing each $x_i$ with $x_i - \ell$. We may thus choose our $\epsilon(z) = \lambda$ uniformly, since given $z \in H$ we may choose some ball of $B \in \mathscr{B} + \mathscr{L}$ which contains $B_\lambda(z)$, and then use the same $x_i$ in either Equation \ref{eq: qc1} or \ref{eq: qc2} for the ball $B_\lambda(z)$ as is used for the ball $B$. \end{proof}

The following proposition will be helpful in establishing that acceptance domains for quasicanonical schemes may be identified with regions cut by hyperplanes. Notice that the sets $Z_i$ below are a similar form to the acceptance domains, as defined in Equation \ref{eq: acceptance domain}. It should be interpreted as saying that such subsets can be used to precisely cover $H^+$, ignoring points of the boundary $H$, which are singular by Lemma \ref{lem: qc => face covering}. We denote the interior of $H^+$ by $\iH^+ = H^+ \setminus H$.

\begin{proposition} \label{prop: halfspace covering} For a quasicanonical scheme we have the following. For each $H^+ \in \sH^+$ we may write
\[
\iH^+ = \left(\bigcup_{i = 1}^\infty Z_i\right) \setminus H,
\]
where each $Z_i$ is equal to an intersection
\[
Z_i = \left(\bigcap_{a \in A_i} W-a\right) \cap \left(\bigcap_{b \in B_i} \W-b\right),
\]
for $A_i$, $B_i \subseteq \Gamma_<$ finite sets for each $i \in \N$. Moreover, this collection of sets may be chosen to be locally finite, in the sense that there is some $\kappa \in \N$ for which each element of $\intl$ is contained in at most $\kappa$ sets of the form $W-a$ or $W-b$ for $a \in A_i$ and $b \in B_i$, $i \in \N$.
\end{proposition}

\begin{proof} According to the previous lemma we may take $\epsilon > 0$ to not depend on $z$ in Equations \ref{eq: qc1} and \ref{eq: qc2}. By density of $\Gamma_<$, we may choose a finite subset $Q \subseteq \Gamma_<$ for which the intersection
\[
B \coloneqq \bigcap_{q \in Q} W - q
\]
has diameter less than $\epsilon$ but has non-empty interior, so contains some $\lambda$-ball.

For $z \in H$, choose $v(z) \in \Gamma_<$ so that $B-v(z)$ contains a $\lambda/2$-neighbourhood of $z$. Choose $x_i \in \Gamma_<$ according to those given by either Equation \ref{eq: qc1} or \ref{eq: qc2}. In the former case we denote the set of $x_i$ by $A(z)$ and let
\[
Z(z) = \bigcup_{a \in A(z)} (B-v(z)) \cap (W-a)
\]
and in the latter case we denote the set of $x_i$ by $B(z)$ and let
\[
Z(z) = (B-v(z)) \cap \left(\bigcap_{b \in B(z)} \W - b\right).
\]
In the former case, by Equation \ref{eq: qc1}, we have that $Z(z) \subseteq H^+$, since $B-v(z) \subseteq B_\epsilon(z)$, as $z \in B-v(z)$ and this set has diameter less than $\epsilon$. Similarly, $Z(z) \supseteq H^+ \cap B_{\lambda/2}(z)$, as $B-v(z)$ was assumed to contain a $\lambda/2$ neighbourhood of $z$ and $B-v(z)$ has diameter less than $\epsilon$ so that Equation \ref{eq: qc1} applies. In the latter case, we have instead that
\[
(H^+ \setminus H) \cap B_\epsilon(z) = \left(\bigcap_{b \in B(z)} W^c - b \right) \cap B_\epsilon(z)
\]
by Equation \ref{eq: qc opp}. By an essentially identical argument to above, we see that $Z(z) \subseteq (H^+ \setminus H)$ and $Z(z) \supseteq (H^+ \setminus H) \cap B_{\lambda/2}(z)$. So in either case
\[
((H^+ \setminus H) \cap B_{\lambda/2}(z)) \subseteq Z(z) \subseteq H^+.
\]

By Lemma \ref{lem: qc => lattice}, $\Gamma_< \cap V(H)$ contains a lattice $\mathscr{L}$ for $V(H)$. By a similar argument to the proof of the previous lemma, we may choose a finite collection of points $z_1$, $z_2$, \ldots, $z_n \in H$ so that the balls $B_{\lambda/2}(z_i + \ell)$, for $\ell \in \mathscr{L}$, cover some $\delta$-neighbourhood of $H$ in $\intl$. For those $z_i$ using Equation \ref{eq: qc1}, we have that
\[
Z(z_i) + \ell = \bigcup_{a \in A(z_i)} \left(\left(\bigcap_{q \in Q} W-q-v(z_i)+\ell\right) \cap (W-a+\ell)\right)
\]
and for those $z_i$ needing Equation \ref{eq: qc2} we have that
\[
Z(z_i) + \ell = \left(\bigcap_{q \in Q} W-q-v(z_i)+\ell\right) \cap \left(\bigcap_{b \in B(z_i)} W^c-b+\ell\right).
\]
In either case we have a finite union of intersections of the desired form. They are contained in $H^+$ and give a locally finite covering of $H^+$ restricted to a $\delta$-neighbourhood of $H$, at least after removing $H$. By density of $\Gamma_<$, we can take further translates by $\Gamma_<$ of $W$ to extend this to a locally finite covering of the whole of $H^+ \setminus H$.
\end{proof}

\section{Cut regions} \label{sec: cut regions}

In this section we discuss a way of analysing the complexity of cut and project sets, not through acceptance domains, but through cut regions, defined below. In Section \ref{sec: complexity} we will prove complexity bounds that hold under the quasicanonical condition, and in Section \ref{sec: removing the quasicanonical condition} we analyse acceptance domains directly, without the intermediate step given by Corollary \ref{cor: acceptance <-> cuts}. The reason for including the discussion on cut regions under the quasicanonical condition is two-fold. Firstly, there are some instances in the literature where it is mistakenly assumed that the conclusion of Corollary \ref{cor: acceptance <-> cuts} holds, and we wanted to provide full proofs in a very general setting where this corollary is indeed true. In Section \ref{sec: troublesome scheming} we give examples to show that Corollary \ref{cor: acceptance <-> cuts} does indeed fail under the commonly used almost canonical condition. This will hopefully help clear some of the misunderstandings surrounding the analysis of cut and project sets, and gives a reference for the correct proofs. The second reason is that there are other situations where it is extremely useful to link cut regions to acceptance domains.

Recall the `slab' of lattice points $\cS(r)$ from Equation \ref{eq: slab}. We similarly define the `box' of lattice points:
\begin{equation}\label{eq: box}
\cB(r) \coloneqq \{\gamma \in \Gamma \mid \|\gamma_\vee\|, \|\gamma_<\| \leq r\}.
\end{equation}
Note that there are constants such that $\cB(c_1 r) \subseteq \Gamma \cap B_r \subseteq \cB(c_2 r)$, so one might prefer to imagine instead the ball of lattice points within radius $r$ of the origin; using $\cB(r)$ will simply be more convenient for the proofs to follow. As above we let $\cS_<(r) \coloneqq \pi_<(\cS(r))$ and similarly $\cB_<(r) \coloneqq \pi_<(\cB(r))$.

Consider the translates of the hyperplanes $H \in \sH$ by the elements of $\cB_<(r)$. Removing these hyperplanes from $W$ cuts the window into various connected components which we call a {\bf cut region}. We denote the set of them by $\sC(r)$. We similarly define $\sC'(r)$, removing instead $\cS_<(r)$ translates. We aim to approximate the acceptance domains $\sA(r)$ using the cut regions $\sC(r)$. The cut regions $\sC'(r)$ will serve as intermediates in establishing this connection.

\begin{definition} Given two collections $\sF_1$ and $\sF_2$ of subsets of ${\intl}$, we say that {\bf $\sF_1$ refines $\sF_2$} and write $\sF_1 \preceq \sF_2$ if for all $F_1 \in \sF_1$ we have $F_1 \subseteq F_2$ for some $F_2 \in \sF_2$. \end{definition}

\begin{proposition} \label{prop: acc <-> cuts} For a polytopal cut and project scheme we have that:
\begin{enumerate}
	\item $\sC(r) \preceq \sC'(r)$ for sufficiently large $r$;
	\item $\sC'(r) \preceq \sA(r)$ for all $r$.
\end{enumerate}
Moreover, if the scheme is quasicanonical, then
\begin{enumerate}\setcounter{enumi}{2}
	\item there exists $c > 0$ for which $\sA(r+c) \preceq \sC'(r)$ for sufficiently large $r$;
	\item there exist $c',\kappa > 0$ for which $\sC'(\kappa r + c') \preceq \sC(r)$ for all $r$.
\end{enumerate}
So for a quasicanonical scheme we have that for sufficiently large $r$:
\[
\sC(\kappa r + c + c') \preceq \sC'(\kappa r + c + c') \preceq \sA(\kappa r + c + c') \preceq \sC'(\kappa r + c') \preceq \sC(r).
\]
\end{proposition}

The above proposition shows that for quasicanonical schemes, so long as we do not mind introducing linear factors, we may replace study of the acceptance domains with the study of the simpler cut regions. Its main conclusion may be restated as follows:

\begin{corollary} \label{cor: acceptance <-> cuts} Consider a quasicanonical cut and project scheme. There exist $\kappa, c>0$ for which the following holds for sufficiently large $r$:
\[
\sC(\kappa r + c) \preceq \sA(\kappa r + c) \preceq \sC(r).
\]
That is, for sufficiently large $r$:
\begin{enumerate}
	\item if the $r$-patches at $x$, $y \in \cps$ disagree then $x^*$ and $y^*$ are separated by a hyperplane $H \in \sH$ translated by an element of $\cB_<(r)$;
	\item conversely, if the $(\kappa r+c)$-patches at $x$, $y \in \cps$ agree, then these points are not separated by any such hyperplane.
\end{enumerate}
\end{corollary}

Before proving Proposition \ref{prop: acc <-> cuts} we introduce some important notation:

\begin{definition} \label{def: stabiliser} For a polytopal cut and project scheme we define the {\bf stabiliser} of $H \in \sH$ as
\[
\Gamma^H \coloneqq \{\gamma \in \Gamma \mid H = H + \gamma_<\}.
\]
\end{definition}

Alternatively, $\Gamma^H = \Gamma \cap (V(H)+\phy)$ and $\Gamma_<^H \coloneqq \pi_<(\Gamma^H)$ is given by $\Gamma_< \cap V(H)$ (recall that $V(H) = H - H$ is the vector subspace parallel to $H$). As each $\Gamma^H$ is a subgroup of $\Gamma \cong \Z^k$, it is free Abelian group of some rank $\rk(H) \coloneqq \rk(\Gamma^H)$. By Lemma \ref{lem: qc => lattice}, $\Gamma_<^H$ spans $V(H)$ for a quasicanonical scheme so $\rk(H) \geq k-d-1$. As we shall see in the next section, we need these ranks to be as large as possible to obtain cut and project sets of low complexity.

The next statement is the main conclusion of this section. It allows us to bound the complexity function in terms of the numbers of cut regions, in the quasicanonical case. It is an immediate corollary of Corollaries \ref{cor: acceptance domains} and \ref{cor: acceptance <-> cuts}, since each $r$-patch corresponds to an acceptance domain of $\sA(r)$.

\begin{corollary} \label{cor: complexity versus cuts} For a quasicanonical cut and project scheme there exist $\kappa, c > 0$ for which, for sufficiently large $r$,
\[
\# \sC(r) \leq p(\kappa r+c) \leq \# \sC(\kappa r+c).
\]
\end{corollary}

\subsection{Proof of Proposition \ref{prop: acc <-> cuts}}

\subsubsection{Overview} The proofs of (1), that $\sC(r) \preceq \sC'(r)$ for sufficiently large $r$, and (2), that $\sC'(r) \preceq \sA(r)$ for all $r$, are trivial.

The proof of (3), that $\sA(r+c) \preceq \sC'(r)$ for sufficiently large $r$, runs as follows. A cut region of $\sC'(r)$ is the intersection of translates of (interiors of) half-spaces from $\sH^+$, or their opposites, from vectors of $\cS_<(r)$. Using Proposition \ref{prop: halfspace covering} we replace these half-spaces with unions of intersections like those defining the acceptance domains. Since vectors of $\cS_<(r)$ are within a bounded radius of the origin, we only need to consider finitely many elements of $\Gamma_<$ defining these regions, which leads to the value of $c$.

The proof of (4), that $\sC'(\kappa r + c') \preceq \sC(r)$, uses the fact that each $\Gamma^H_<$ contains a lattice. So a translated hyperplane defining a cut region of $\sC(r)$ can be replaced with a translate of $\Gamma_<$ within a bounded distance of the origin. A further argument using the quasicanonical condition shows that this element can be chosen in $W-W$.

\begin{proof}[Proof of (1)]
For $r$ large enough so that $W-W \subseteq B_r$, we have that $\cS(r) \subseteq \cB(r)$. Therefore for such $r$, elements of $\sC(r)$ are defined by making extra cuts to those from $\sC'(r)$, so $\sC(r) \preceq \sC'(r)$. \end{proof}

\begin{proof}[Proof of (2)]
Let $A \in \sA(r)$. Any cut region $C \in \sC'(r)$ containing $x_0 \in A$ is wholly contained in $A$. Indeed, by Equation \ref{eq: acceptance domain}, the boundary of $A$ is contained in a union of $\cB_<(r)$ translates of $\partial W$. We have that $\partial W \subseteq \bigcup_{\sH} H$, so $\partial A$ is contained in the union of cut hyperplanes translated by the elements of $\cB_<(r)$. Any other $y \in C$ is connected to $x_0 \in C$ by a path which avoids $\partial C$, and therefore avoids $\partial A$, and so $y \in A$ too, hence $C \subseteq A$. \end{proof}

\begin{proof}[Proof of (3)]
For each $H \in \sH$, choose a finite number of sets $Z_i(H) \subseteq H^+$, for $i=1,\ldots,n(H)$, satisfying the following:
\begin{enumerate}
	\item each $Z_i(H)$ is a finite intersection of $\Gamma_<$ translates of $W$ and $W^c$;
	\item the union $Z(H) = \bigcup_{i=1}^{n(H)} Z_i(H)$ satisfies $(Z(H) \setminus H) \cap (W+(W-W)) = \iH^+ \cap (W+(W-W))$.
\end{enumerate}
Such a covering exists by Proposition \ref{prop: halfspace covering}. Let $c$ be the maximum of the norms of the elements of $\Gamma_<$ needed in constructing all of these sets, once lifted and then projected to $\phy$. It easily follows that
\begin{equation} \label{eq: union of Z}
(\iH^+ + \gamma) \cap W = ((Z(H) \setminus H) +\gamma) \cap W
\end{equation}
and
\begin{equation} \label{eq: union of Z^c}
((H^+)^c+\gamma) \cap W = ((Z(H)^c \setminus H)+\gamma) \cap W
\end{equation}
for any $\gamma \in W-W$.

We claim that $\sA(r+c) \preceq \sC'(r)$. So take $C \in \sC'(r)$ and let $x_0 \in C$ be a non-singular point of $C$. We may write $C \subseteq W$ as a finite intersection
\begin{equation} \label{eq: cut half-spaces}
C = \bigcap_X (X + \gamma(X)_<),
\end{equation}
where each $X = \iH^+$ or $X = (H^+)^c$ for some $H^+ \in \sH^+$, and each $\gamma(X) \in \cS(r)$. This implies that each $\gamma(X)_< \in W-W$. So for each $X$, there is some $H \in \sH$ for which we may write
\[
(X + \gamma(X)_<) \cap W = (\iH^+ + \gamma(X)_<) = ((Z(H) \setminus H) + \gamma(X)_<) \cap W,
\]
by Equation \ref{eq: union of Z} or, by Equation \ref{eq: union of Z^c},
\[
(X + \gamma(X)_<) \cap W = ((H^+)^c + \gamma(X)_<) = ((Z(H)^c \setminus H) + \gamma(X)_<) \cap W.
\]

In the former case, choose some $i$ so that $x_0 \in Z_i(H) + \gamma(X)_<$. Note that every point of $H + \gamma(X)_<$ is singular by Lemma \ref{lem: qc => face covering}, so there is such an $i$, as the union of the shifts of the $Z_i(H)$ covers the open half-space $X + \gamma(X)_<$ which contains $x_0$. Replacing each term $W$ with $\iW$ in the intersection defining $Z_i(H)$, this is an open subset which is contained in $X+\gamma(X)_<$. So we have found an intersection of translates under $\Gamma_<$ of $\iW$ and $W^c$ which is contained in the half-space $X + \gamma(X)_<$ and contains $x_0$, where each translation vector is the sum of $\gamma(X)_<$ and one of the finite number of translates defining the $Z_i(H)$.

In the latter case, we may write:
\[
(X + \gamma(X)_<) \cap W = \left(\bigcap_{i=1}^{n(H)} (Z_i(H)^c \setminus H) + \gamma(X)_<\right) \cap W.
\]
Since $x_0$ belongs to the set on the left it also belongs to each term $Z_i(H)^c + \gamma(X)_<$ appearing on the right, which we may write as
\[
Z_i(H)^c + \gamma_<(X) = \left(\bigcup W^c - a+ \gamma(X)_<\right) \cup \left(\bigcup W - b +\gamma(X)_<\right).
\]
So $x_0$ belongs to one of the terms in the union above, which we may take as either of the form $W^c - a + \gamma(X)_<$ or, since $x_0$ is non-singular, $\iW - b + \gamma_<(X)$. Making such a choice for each $i$ and taking the intersection results in a subset containing $x_0$ and contained in $X + \gamma(H)_<$.

So in either case, for each $X$ in the intersection of Equation \ref{eq: cut half-spaces} we may replace each $X+\gamma(X)_<$ with a smaller set containing $x_0$ and given as an intersection of $\Gamma_<$ translates of $W$ or $W^c$. Each such translate was the sum of an element of norm at most $c$ when lifted then projected to $\phy$, and an element $\gamma_<(X)$ with $\gamma(X) \in \cS(r)$, so has norm at most $r$ when projected to $\phy$. So the shifts of $W$ and $W^c$ in the intersection are of projections of lattice elements which project in $\phy$ to elements of norm at most $r+c$. Moreover, we may assume that their projections to $\intl$ belong to $W-W$, since otherwise they may be omitted without effecting the intersection with $W$. So these translation vectors belong to $\cS_<(r)$ and so their intersection is contained in an acceptance domain, which in turn is contained in $C$, as required.
\end{proof}

\begin{proof}[Proof of (4)]
We must show that every $\cB_<(r)$ translate of a hyperplane contributing to a cut region is an $\cS_<(\kappa r + c')$ translate of a hyperplane, for some fixed $c',\kappa > 0$. So let $H \in \sH$ and $\gamma \in \cB(r)$, with $H + \gamma_<$ intersecting the window non-trivially.

By Lemma \ref{lem: qc => lattice}, $\Gamma_<^H$ spans $V(H)$. So take a basis for $V(H)$ of elements from $\Gamma_<^H$ with fundamental domain of diameter less than $c_1$ in $\intl$ (which, with $c_1$ large enough, we may take as such a bound for some basis for all choices of $H$). Hence we may find a $\Z$-linear sum $s \in \Gamma^H$ of our chosen basis elements for which $\|\gamma_< - s_<\| \leq c_1$, so $\|s_<\| \leq r + c_1$. Now $H+(\gamma_<-s_<) = H+\gamma_<$; we would like to replace $\gamma$ with $\gamma - s$ in defining the cut $H + \gamma_<$.

In $\intl$, we have seen that $\|\gamma_< - s_<\| \leq c_1$. In $\phy$, it is not hard to see that $\|\gamma_\vee - s_\vee\| \leq \kappa r + c_2$ for some constants $c_2,\kappa > 0$. Indeed, $s_<$ was chosen relative to some basis $(\gamma_1)_<$, \ldots, $(\gamma_{k-d-1})_<$ of $V(H)$ of elements in $\Gamma_<^H$. Writing $s_< = \sum k_i \cdot (\gamma_i)_<$, since $\|s_<\| \leq r+c_1$, we thus have that $\sum |k_i| \leq \kappa_1(r+c_1)$, with $\kappa_1$ depending on the lengths of the basis vectors $(\gamma_i)_<$. So $s = \sum k_i \cdot \gamma_i$ with $\|s\| \leq \kappa_2(\sum |k_i|) \leq \kappa_1\kappa_2(r+c_1)$, with $\kappa_2$ depending on the lengths of the $\gamma_i \in \Gamma$. By assumption $\gamma \in \cB(r)$, so $\|\gamma_\vee\| \leq r$ and hence
\[
\|(\gamma - s)_\vee\| = \|\gamma_\vee - s_\vee\| \leq \kappa_3\|\gamma - s\| \leq \kappa_3\|\gamma\| + \kappa_3\|s\| \leq \kappa_3 r + \kappa_1\kappa_2\kappa_3(r+c_1),
\]
with $\kappa_3$ depending only on the projection $\pi_\vee$. So replacing $\gamma$ with $\gamma - s$, we may assume that $\|\gamma_<\| \leq c_1$ and that $\|\gamma_\vee\| \leq \kappa r + c_2$, with $\kappa \coloneqq \kappa_3(1+\kappa_1\kappa_2)$ and $c_2 \coloneqq \kappa_1\kappa_2\kappa_3c_1$.

For each hyperplane $H \in \sH$, let $X_H \subseteq H$ be a bounded region containing all points of $H$ within radius $c_1$ of $W$. By Lemma \ref{lem: qc => face covering} one may find a finite number of translates $W + (s_i)_<$, for $s_i \in \Gamma$, whose faces parallel to $H$ cover $X_H$ so that either $H+(s_i)_< = H$ or $H^\op + (s_i)_< = H$, where $H^\op \in \sH$, if it exists, is parallel to $H$ (that is, with $V(H) = V(H^\op)$).

We have that $H + \gamma_<$ intersects $W$, and since $\|\gamma_<\|\leq c_1$, we have that $X_H + \gamma_<$ intersects $W$ by the choice of $X_H$. Hence there is some translate $W + (s_i)_< + \gamma_<$ intersecting $W$, which implies that $(s_i)_< + \gamma_< \in W-W$, and with $H + \gamma_< = H+((s_i)_<+\gamma_<)$ or $H + \gamma_< = H^\op + ((s_i)_<+\gamma_<)$. In either case, we may replace the cut defined by $H + \gamma_<$ by a cut defined by $H' + (s_i+\gamma)_<$ with $H' \in \sH$. As established, $(s_i + \gamma)_< \in W-W$, and since the $s_i$ were chosen from a finite collection we have that $\|(s_i)_\vee + \gamma_\vee\| \leq c_3 + \|\gamma_\vee\| \leq c_3 + (\kappa r + c_2)$, where we choose $c_3$ so that each $\|(s_i)_\vee\| \in B_{c_3}$. Setting $c \coloneqq c_2 + c_3$, we thus have that $(s_i+\gamma)_< \in \cS_<(\kappa r + c)$, and so the result follows. \end{proof}

\begin{remark} We comment on the size of the constant $\kappa$ in Proposition \ref{prop: acc <-> cuts}. 

\vspace{0.2cm}
\noindent (a) In the codimension 1 case the window is an interval and both the acceptance domains $\cA(r)$ and cut regions $\cC'(r)$ are precisely the intervals in $W$ between the $\cS(r)_<$ translates of the two endpoints of $W$. If $z \in \cC(r)$ is such that $z+h \in W$ for one of the endpoints $h$ (so is actually cutting the window) then $z \in W-h \subset W-W$ and thus $z \in \cC(r)$ anyway. Hence, in codimension 1, $\cA(r) = \cC'(r)$, and $\cC'(r) = \cC(r)$ for sufficiently large $r$. In particular we may take $\kappa = 1$ in Proposition \ref{prop: acc <-> cuts}.

\vspace{0.2cm}
\noindent (b) In codimension larger than 1, we necessarily have that $\kappa > 1$. We give here a brief geometric indication of why.
Distinct cuts made by a hyperplane $H$ are identified with the cosets of $\Gamma^H$ in $\Gamma$. In particular, a given cut $H+z$ contributes to $\cC'(r)$ when the slab $\cS(r)$ intersects the coset $z+\Gamma^H$, and similarly the cut contributes to $\cC(r)$ when the box $\cB(r)$ intersects this coset. If it were the case that $\Gamma^H \subseteq \intl$, then these cuts  would be made simultaneously (at least after $r$ is taken sufficiently large so that $W-W \subseteq B_<(r)$, where we define $B_<(r)$ and $B_\vee(r)$ to be the balls of radius $r$ in $\intl$ and $\phy$, respectively). However, we have chosen $\intl$ to not contain lattice points, that is with $\pi_<$ injective on $\Gamma$. So the cosets of $\Gamma^H$ must be `slanted' in a certain sense with respect to $\intl$. Therefore, the box $B_\vee(r) + B_<(r)$ will intersect these cosets --- visualised as cutting through near the edge or corner of a box --- at faster constant rate to the slab $B_\vee(r) + (W-W)$, where $r$ is given by the distance from the origin at which the coset passes through a small strip containing the physical space.

The proof of (4) works by replacing a representative of a coset with another, whose $\intl$ coordinate is in $(W-W)$. This exchanges distance in the internal space with that of the physical space. By rescaling the metric in the internal space direction, the cut and project set and its geometry remain the same, and one may make $\kappa$ as close to $1$ as desired, but never equal to $1$ unless working in a setting which allows $\pi_<$ to be non-injective on $\Gamma$.
\end{remark}

\subsection{Troublesome scheming} \label{sec: troublesome scheming}

The result of Corollary \ref{cor: acceptance <-> cuts}, although stated there differently, is the tool implicitly utilised in \cite{Jul10} when analysing the complexity function of a cut and project set. Unfortunately there is an error in the proof of \cite[Proposition 3.1]{Jul10} (discussed in more detail in \cite[p.\ 73--74]{HaynKoivSaduWalt2015}). We demonstrate below in Example \ref{ex: triangle} that the almost canonical condition is not strong enough for this corollary to hold. 

\begin{definition} A polytopal cut and project scheme is {\bf almost canonical} if for each $H \in \sH$ with associated face $F \subseteq \partial W$, we have that $F + \Gamma_< \supseteq H$. \end{definition}

This first example shows that the quasicanonical condition does not imply the almost canonical condition. 

\begin{example} Let $k=4$, $d=2$ and $\intl = \{0\}^2 \times \R^2$, which we identify with $\R^2$ in the obvious way. Let $W = [0,1+\alpha] \times [0,1+\beta]$ and $f_1 = (1,1+\beta)$, $f_2 = (1,-1-\beta)$, $f_3 = (1+\alpha,1)$, $f_4 = (-1-\alpha,1)$, where $0 < \alpha, \beta < 1$. For suitable $\alpha$ and $\beta$, we may choose $\Gamma$ so that the projected lattice given by the $\Z$-span $\Gamma_< = \langle f_1, f_2, f_3, f_4 \rangle_\Z$, is rank $4$ and dense in $\intl$. The quasicanonical condition holds, and corresponds to the middle diagram of Figure \ref{fig:quasicanonical}. But we may only cover each $H\in V(\sH)$ with $\Gamma_<$ translates of $\partial W$ in an alternating fashion and not using the same face, so the almost canonical condition fails. \end{example}

The point here is that different faces of the window need to be used to cover $H$. If one modifies the almost canonical condition to allow a potentially opposite face to also be used in covering $H$, then any quasicanonical scheme must satisfy this modified condition by Lemma \ref{lem: qc => face covering}.

In the next example we illustrate that the quasicanonical condition is also not weaker than the almost canonical condition and, furthermore, that the almost canonical condition is not enough for the statement of Corollary \ref{cor: acceptance <-> cuts} to hold. 

\begin{example} \label{ex: triangle} Consider the case $k=3, d=1$ and let $\intl = \{0\}^1 \times \R^2$. Choose two linearly independent points $x$, $y \in \intl$. Let $W$ be the triangle with vertices $0$, $x$ and $y$. Then $V(\sH) = \{\langle x \rangle_\R, \langle y \rangle_\R, \langle y-x \rangle_\R\}$. We may arrange that the scheme is such that $\Gamma_< = \langle x,y,z \rangle_\Z$ for some $z \in \intl$ chosen `irrationally' i.e., with $\rk(\Gamma_<) = 3$, and with $\Gamma_<$ dense in $\intl$. We may assume that $z\in W$, see Figure \ref{fig: triangle}. 

\begin{figure}
	\def\svgwidth{0.4\textwidth}
	\centering
	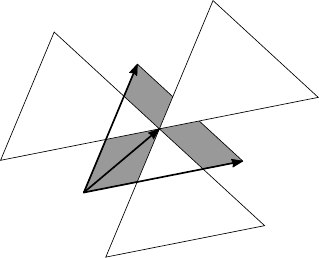
	\caption{The window $W$ of Example \ref{ex: triangle} with three translates, by $z, z-x$ and $z-y$.}\label{fig: triangle}
\end{figure}

Essentially by construction this scheme is almost canonical: for the hyperplane $H = \langle \gamma \rangle_\R$ one can use the translates in $\Z \cdot\gamma$ of the face $F \subseteq H$ to cover $H$. But it is not quasicanonical. The issue is the same as in the rightmost picture in Figure \ref{fig:quasicanonical}. 

In more detail, let $c>0$ be such that $z \in \Gamma_<(c)$. We claim that $\sA(R) \preceq \sC(c)$ does not hold for \emph{any} $R > 0$. There are three $\Gamma_<$ translates of $W$, by $z$, $z-x$ and $z-y$, as in Figure \ref{fig: triangle} with a vertex at $z$. Every other translate either does not contain $z$ or contains it in its interior. Indeed, we may write $\gamma_< = n_1 x + n_2 y + n_3 z$ for $\gamma \in \Gamma$. It is easily checked that if $n_3 = 1$ then either $z \notin W + \gamma_<$ or $\gamma = z$, $z-x$ or $z-y$. For $n_3 \neq 1$, the translates of the hyperplanes defining the boundary of $W$ do not contain $z$ by the irrationality of $z$ relative to $(x, y)$. In summary, the translated window either contains $z$ in its interior or exterior, or is one of the three white triangles as in Figure \ref{fig: triangle}.

Consider three vectors $\alpha_i$ so that the vectors $z+\alpha_i$ belong to the three different shaded regions in Figure \ref{fig: triangle}. Then for $0 < \epsilon < 1$, the three points $z + \epsilon \alpha_i$ are separated by a hyperplane of $\sH + z$, so these three points belong to the different cut regions of $\sC(c)$. On the other hand, for any $R>0$, for sufficiently small $\epsilon$ these three points belong to the same acceptance domain of $\sA(R)$. Indeed, none of them belong to the three white triangles of Figure \ref{fig: triangle}, and every other translate of $W$, as has been established, contains $z$ in its interior or exterior, so we may choose $\epsilon$ small enough so the $\alpha_i$ are close enough to $z$ so that each of the finite number of translates defining $\sA(R)$ contain all three points or none of them.

We see that in this example, despite the scheme being almost canonical, we may not replace the analysis of acceptance domains with that of cut regions.
\end{example}

\section{Complexity functions for quasicanonical cut and project sets} \label{sec: complexity}

In this section we calculate the asymptotic growth rate of the complexity function in the quasicanonical case, adapting ideas from \cite{Jul10}. We shall refine the argument in Section \ref{sec: removing the quasicanonical condition} to improve the result by removing the quasicanonical requirement. By Corollary \ref{cor: complexity versus cuts}, in the quasicanonical case the complexity function can be analysed by counting cut regions. To this end we need to consider the $\cB_<(r)$ translates of the hyperplanes of $\sH$ intersecting the window. There are of order $r^k$ lattice points in $\cB(r)$ and so, without further restrictions, one would generically expect that there are of order $r^{k-1}$ translates of $H \in \sH$ intersecting $W$. However, in higher codimensions the low complexity cut and project sets are not generic: they are schemes where different $\Gamma_<$ translates of the hyperplanes can coincide, reducing the number of cuts and so also the complexity.

We shall count cut regions by counting their vertices, which may be estimated by considering `flags' of the linear subspaces defining the boundary of $W$.

\begin{definition} \label{def: flag} Let $X$ be an $n$-dimensional vector space, and $f = \{V_1,\ldots,V_n\}$ be a set of codimension one (i.e., $(n-1)$-dimensional) linear subspaces of $X$. If $\bigcap V_i = \{0\}$ then we call $f$ a {\bf flag}. \end{definition}

We emphasise that there should be precisely $n$ subspaces constituting a flag in an $n$-dimensional space and no more. For ease of reference, we state two basic facts related to flags, whose proofs are a matter of basic linear algebra and are omitted:

\begin{lemma} \label{lem: simple flags} Let $V$ be a codimension $1$ subspace of the $n$-dimensional vector space $X$, and $V'$ any subspace of $X$. Then either $V' \subseteq V$ or $\dim(V \cap V') = \dim(V')-1$. Hence, any collection $V_1,\ldots,V_N$ of codimension $1$ subspaces of $X$ with $\bigcap V_i = \{0\}$ contains a subset which is a flag. Moreover, we can choose such a flag to contain any given $V_i$. \end{lemma}

\begin{lemma} \label{lem: singleton intersections} Let $V_1$, \ldots , $V_n$ be codimension one subspaces of $X$, where $\dim(X) = n$. If $f = \{V_1, \ldots, V_n\}$ is a flag then $\bigcap V_i + c_i$ is a singleton set for any $c_1$, \ldots, $c_n \in X$. Conversely, if $\bigcap V_i + c_i$ is a singleton set for some $(c_i)$, then $f$ is a flag. \end{lemma}

The lemma above allows us to also say that a collection $f = \{H_1,\ldots,H_n\}$ of affine hyperplanes of some $X \cong \R^n$ is a {\bf flag} if their intersection is a single point, which is if and only if the collection $V(f)$ is a flag as defined previously in Definition \ref{def: flag}. 

Let us briefly allow ourselves to consider non-compact versions of polytopes (whose interiors appear as certain cut regions later):

\begin{definition} An intersection of a finite number of half-spaces of $\R^d$ with non-empty interior is called a {\bf convex polyhedron}. If it is compact then it is called a {\bf convex polytope}. \end{definition}

We drop the adjective \emph{convex} from the above definitions, since all polyhedra/polytopes here will be convex. The finite collection of half-spaces defining a polyhedron will always be denoted by $\sH^+$, and the associated hyperplanes by $\sH$.

\begin{definition} A {\bf face} of an $n$-polyhedron $P$ is a subset $H \cap P$ where $H$ is a hyperplane for which $P$ is contained in one of the closed half-spaces of $H$. We call it an {\bf $n$-face} if $\dim (H \cap P) = n$. We call a $0$-face a {\bf vertex}, a $1$-face an {\bf edge} and an $(n-1)$-face a {\bf facet}. We will regularly identify a vertex $\{v\}$ with the point $v$ itself. \end{definition}

Here are some well-known properties of the face structure of polyhedra. A good reference for this is \cite{BG09}:

\begin{proposition}\label{prop: polyhedra} Let $P$ be an $n$-dimensional polyhedron with irredundant set $\sH^+$ of defining half-spaces (which means that omitting any $H^+ \in \sH^+$ results in an intersection of half-spaces different to $P$).
\begin{enumerate}
	\item The set $\sH^+$ is uniquely determined by $P$.
	\item Every $j$-face of $P$ is an intersection of $(n-j)$ hyperplanes of $\sH$, which intersect to a $j$-dimensional affine subspace, intersected with $P$. In particular, every facet is the intersection $H \cap P$ for $H \in \sH$.
	\item The boundary $\partial P$ is the union of facets of $P$.	
	\item The $(j-1)$-faces of $P$ are the facets of the $j$-faces.
	\item Let $F_0 \subset F_1 \subset \cdots \subset F_m = P$ be a strictly ascending maximal chain of non-empty faces of $P$. Then $\dim(F_{i+1}) = \dim(F_i) + 1$ for each $i$. Moreover, $\dim(F_0) = \dim(U)$ where $U$ is the intersection of linear subspaces $V(H)$ for $H \in \sH$. In particular, $P$ has vertices if and only if the subspaces $V(H)$ for $H \in \sH$ contain a flag.
\end{enumerate}
\end{proposition}

Recall the definition of the stabiliser subgroup $\Gamma^H \leqslant \Gamma$ from Definition \ref{def: stabiliser}. The main theorem of this section is the following. 

\begin{theorem}\label{thm: complexity} Suppose a quasicanonical $k$-to-$d$ cut and project scheme is given. Consider the collection $\sF$ of flags of $\sH$. Then $p(r) \asymp r^\alpha$ where
\[
\alpha = \max_{f \in \sF} \alpha_f, \ \text { for }  \ \alpha_f \coloneqq \sum_{H \in f} (k - \mathrm{rk}(H)-1).
\]
\end{theorem}

\subsection{Proof of Theorem \ref{thm: complexity}}

Lemma \ref{lem: count components} below estimates the number of cut regions obtained by removing translates of codimension $1$ hyperplanes from $\intl$. This lemma is similar to \cite[Lemma 2.7]{Jul10}, but we prove it using a different method, by counting vertices of connected components.

\begin{lemma} \label{lem: count components} Consider a collection $V_1$, $V_2$,\ldots, $V_N$ of codimension $1$ linear subspaces of $\intl$, with $\dim(\intl) = n$, and suppose that $\bigcap V_i = \{0\}$. Let $\sF$ be the set of flags of the $V_i$. Take a collection $H_i^1$, $H_i^2$,\ldots, $H_i^{h(i)}$ of translates of each $V_i$ and consider the set $\sC$ of connected components of the complement of their union in $\intl$. Then there are constants $C_1$, $C_2$, determined by only the collection $\{V_i\}$, for which
\[
\left(C_1 \cdot \max_{f \in \sF} \prod_{j=1}^n h(f_j)\right) \leq \# \sC \leq \left(C_2 \cdot \sum_{f \in \sF} \prod_{j=1}^n h(f_j)\right)
\]
where each $f_j \in \{1,\ldots,n\}$ is the index of the $j$th element $V_{f_j} \in f$ of the flag. \end{lemma}

\begin{proof} The closure $C$ of each connected component of $\sC$ is an intersection of half-spaces, bounded by the various translates $H_i^j$. Denote the set of all vertices of connected components by $\mathcal{V}$.

Since $\bigcap V_i = \{0\}$, by Proposition \ref{prop: polyhedra} each $C \in \sC$ has at least one vertex, given as an intersection of some of the hyperplanes defining $C$. At each vertex $v \in \mathcal{V}$ there can be at most a bounded number of $C \in \sC$ containing $v$, since at most $N$ hyperplanes $H_i^j$ can pass through $v$. Since each element of $\sC$ is incident with at least one vertex of $\mathcal{V}$, there is thus some constant $C_1$ for which
\[
\# \sC \geq C_1 \cdot \# \mathcal{V}.
\]
On the other hand, each $C \in \sC$ can have at most a bounded number of vertices (since we have only finitely many $V_i$, and every vertex is the intersection of $n$ of the defining hyperplanes of $C$). Hence, there is some constant $C_2$ for which
\[
\# \sC \leq C_2 \cdot \# \mathcal{V}.
\]
So we are done if we can bound the number of vertices.

By Lemma \ref{lem: simple flags} the set $\sF$ of flags is non-empty. Given any flag $f = \{V_{f_1},\ldots,V_{f_n}\} \in \sF$, by Lemma \ref{lem: singleton intersections} any corresponding set of translates $\{H_{f_1}^{j_1},\ldots H_{f_n}^{j_n}\}$ intersects at a single point, that is, at a vertex $\mathcal{V}$. Moreover, different translates produce different vertices. Indeed, take two different sets of translates which employ different translates of $V_{f_i}$, say with $H_{f_i}^j \neq H_{f_i}^{j'}$. These hyperplanes are parallel, so $H_i^j \cap H_i^{j'} = \emptyset$ and hence the intersections of each collection result in different vertices. So there are precisely $\prod_{j=1}^n h(f_j)$ vertices of $\mathcal{V}$ coming from the flag $f$, hence
\[
\# \mathcal{V} \geq \max_{f \in \sF} \prod_{j=1}^n h(f_j).
\]
We may only claim a maximum here, since some vertices may belong to intersections of translates of different flags.

On the other hand, suppose that $v \in \mathcal{V}$, so that $\{v\} = H_{i_1}^{j_1} \cap H_{i_2}^{j_2} \cap \cdots \cap H_{i_m}^{j_m}$. By Lemmas \ref{lem: simple flags} and \ref{lem: singleton intersections} this may be refined to an intersection of a flag of the $H_i^j$. It follows that
\[
\# \mathcal{V} \leq \sum_{f \in \sF} \prod_{j=1}^N h(f_j).
\]
The result now follows from our above comparisons of $\# \sC$ with $\# \mathcal{V}$. 
\end{proof}

The above lemma allow us to calculate the number of connected components of cut regions from the number of relevant cuts. The lemma below establishes how many relevant cuts there are in a quasicanonical cut and project scheme. The proof follows from the proof of Lemmas \ref{lem: face-cuts} and \ref{lem: nice cuts} and will be postponed until Section \ref{sec: removing the quasicanonical condition}. 

\begin{lemma} \label{lem: hyperplane hitting ball}  Let $H \in \sH$ and $B \subseteq \intl$ be a bounded subset of the internal space with non-empty interior. Let $N(H,B,r)$ denote the number of translates under $\cB_<(r)$ of the hyperplane $H$ which intersect $B$. Then
\[
N(H,B,r) \asymp r^{k - \mathrm{rk}(H) - 1}.
\]
\end{lemma}

The above lemma bounds the number of cuts made by the $\cB(r)$ translates of the hyperplanes to the window. Lemma \ref{lem: count components} then allows us to work out the number of cut regions in the window, which by Corollary \ref{cor: complexity versus cuts} allows us to bound the complexity function, proving Theorem \ref{thm: complexity}:

\begin{proof}[Proof of Theorem \ref{thm: complexity}] First we show that $r^\alpha \ll p(r)$. By Corollary \ref{cor: complexity versus cuts} we need to show that $r^\alpha \ll \# \sC(r)$. Consider the connected components of $\intl$ with $\cB_<(r)$ translates of the hyperplanes removed. Such a connected component either does not intersect $W$, in which case it does not contribute a connected component to $\sC(r)$, or it intersects $W$, in which case it contributes precisely one (by convexity) connected component to $\sC(r)$. Consider a small ball $B \subseteq W$ (for the purposes of visualisation, take it far from the boundary of $W$). It is not hard to see, by basic linear algebra and continuity, that $B$ may be taken small enough so that if $C$ is a cut region bounded by translates $H_1,\ldots,H_n$ which intersect $B$ non-trivially, then $C$ necessarily intersects $W$ and so contributes a connected component to $\sC(r)$ in $W$. Hence we may obtain a lower bound for $\sC(r)$ by considering the number of connected components in $\intl$ given by removing $\cB_<(r)$ translates of the hyperplanes of $\sH$ from $\intl$ which intersect $B$ non-trivially. By Lemma \ref{lem: hyperplane hitting ball} there are $\gg r^{k-\mathrm{rk}(H)-1}$ such hyperplanes. It follows from the lower bound of Lemma \ref{lem: count components} that
\[
p(r) \gg \# \sC(r) \gg \max_{f \in \sF} \prod_{j=1}^{k-d} r^{k-\mathrm{rk}(H_{f_j})-1} = \max_{f \in \sF} r^{\alpha_f} \gg r^\alpha.
\]

We now prove the upper bound $p(r) \ll r^\alpha$. By Corollary \ref{cor: complexity versus cuts} we have that $p(r) \ll \# \sC(r)$. Finding an upper bound for the number of connected components of $\sC(r)$ is simpler: any connected component of $\intl$ with $\cB_<(r)$ translates of hyperplanes removed contributes at most one connected component to $\sC(r)$ (that is, if it is contained in $W$) and has bounding hyperplanes intersecting $W$. So by taking $B=W$ in Lemma \ref{lem: hyperplane hitting ball} and also Lemma \ref{lem: count components},
\[
p(r) \ll \# \sC(r) \ll \sum_{f \in \sF} \prod_{j=1}^{k-d} r^{k-\mathrm{rk}(H_{f_j})-1} = \sum_{f \in \sF} r^{\alpha_f} \ll r^\alpha.
\]
\end{proof}

\section{Removing the quasicanonical condition} \label{sec: removing the quasicanonical condition}
Notice that by Proposition \ref{prop: acc <-> cuts} we have that $\sC(r) \preceq \sA(r)$ for sufficiently large $r$, whether the window is quasicanonical or not; that is, the cut regions are smaller than the acceptance domains. So $p(r) \leq \# \sC(r)$ for sufficiently large $r$. Hence an upper bound on the number of cut regions may be found just as above, with the result that $p(r) \ll r^\alpha$ with $\alpha$ as given in Theorem \ref{thm: complexity}. However, if we drop the quasicanonical condition then this estimate is not necessarily optimal.

For each $H \in \sH$, find elements of $\Gamma$
\begin{equation} \label{eq: H-basis}
b = (h_1^H,h_2^H,\ldots,h_{\beta_H}^H,f_{\beta_H + 1}^H, f_{\beta_H+2}^H,\ldots,f_{k-d}^H),
\end{equation}
for which $b_<$ is a basis for $\intl$ and each $h_i^H \in \Gamma^H$ with $\beta_H$ as large as possible. Since one can begin constructing $b$ by taking elements in $\Gamma^H$, one sees that $\beta_H$ is given by the dimension of the $\R$-linear span of $\Gamma^H_<$. From this it is clear that the $\R$-span of the vectors $f^H_i$ intersects the $\R$-span of $\Gamma^H_<$ only at $0$.

In this section we shall prove that the correct growth rate of the complexity function for a general polytopal cut and project set is determined by the numbers $\rk(H)$ and $\beta_H$:

\begin{theorem} \label{thm: generalised complexity} Suppose given a polytopal $k$-to-$d$ cut and project scheme. Consider the collection $\sF$ of flags of $\sH$. Then $p(r) \asymp r^{\alpha'}$ where
\[
\alpha' \coloneqq \max_{f \in \sF} \alpha_f', \ \ \mathrm{for} \ \ \alpha_f' \coloneqq \sum_{H \in f} (d - \rk(H) + \beta_H).
\]
\end{theorem}

As one simple consequence, we may bound from above the rate of growth of the complexity of \emph{any} polytopal cut and project set, as below. It is also the case that $p(r)\gg r^d$ when the polytopal cut and project set is aperiodic, but we postpone the proof to the  forthcoming part II of this article. 

\begin{corollary} Any polytopal $k$-to-$d$ cut and project set is such that $p(r) \ll r^{d(k-d)}$. \end{corollary}

\begin{proof} Since $\beta_H \leq \rk(H)$ we have that $\alpha_f' = d - \rk(H) + \beta_H \leq d - \beta_H + \beta_H = d$. Since a flag contains $(k-d)$ elements, we see that $\alpha' \leq d(k-d)$. \end{proof}

Suppose that we set the data of $\tot$, $\phy$, $\intl$ and $W$. Then a generic choice of $\Gamma$ (interpreted, for example, with respect to the Haar measure in the space of lattices) results in an allowed cut and project scheme (i.e., with $\pi_\vee$ injective on $\Gamma$ and $\Gamma_<$ dense), and also with each $\Gamma^H$ trivial (since $\Gamma^H$ is non-trivial precisely when there is a non-zero lattice point in the codimension $1$ set $V(H) + \phy$). So a generic cut and project scheme has `maximal complexity'.

\begin{corollary} For a generic lattice as above, the resulting polytopal cut and project sets have $p(r) \asymp r^{d(k-d)}$. \end{corollary}

\begin{remark}
In both the quasi- and almost canonical cases we have that each $\beta_H = k-d-1$, the dimension of $V(H)$. It follows as in \cite{Jul10} that $p(r)\ll r^{d(k-d)}$ and that $p(r) \gg r^d$ in the aperiodic case (the relevant arguments can be found in Theorem 4.1, also see Remark 2.4). 
\end{remark}

\subsection{Proof strategy}
Let us give an informal explanation for why the exponent of Theorem \ref{thm: generalised complexity} is the correct one. Take a face $\partial_H$ of the window, associated to the hyperplane $H$. We wish to consider which `cuts' $W \cap (\partial_H + \gamma_<)$ can occur as $\gamma$ is ranged over $\cB(r)$.

Since the rank of $\Gamma$ is $k$, there are $\asymp r^k$ choices for $\gamma$. However, loosely speaking, we lose $\rk(H)$ in the exponent because translates differing by an element of $\Gamma^H_<$ are parallel. Given a translated face $\partial_H + \gamma_<$, to ensure that it actually intersects the window we are constrained in $(k-d)$ of the coordinates, which we use to translate the face back to the vicinity of the window. Some number of these constrained coordinates, namely, $\beta_H$ of them, are already taken into account in the number $\rk(H)$ corresponding to parallel cuts. So we obtain the exponent
\[
k - \rk(H) - ((k-d)-\beta_H) = d - \rk(H) + \beta_H
\]
for the relevant number of cuts of the face with the window.

\subsection{Proof of upper bound of Theorem \ref{thm: generalised complexity}}
We now make the above ideas precise. Given a face $\partial_H$ of $W$ corresponding to $H \in \sH$, we call $\partial_H + \gamma_<$ a {\bf face-cut} if $\partial_H + \gamma_<$ intersects $W$ non-trivially. If additionally $\gamma \in \cB(r)$ we write $\gamma \in \Delta_H(r)$. We have the following simple upper bound on the complexity, which we get by extending the face cuts $\partial_H + \gamma_<$ to `full cuts' $H + \gamma_<$.

\begin{lemma} \label{lem: upper bound}Consider the set $\sD$ of connected components of $W$ with $(\Delta_H(r))_<$ translates of the hyperplanes $H \in \sH$ removed. That is, we consider the connected components of $W - [\bigcup_{H \in \sH} H + (\Delta_H(r))_<]$. Then $p(r) \leq \# \sD$. \end{lemma}

\begin{proof} By Equation \ref{eq: acceptance domain} the boundaries of the acceptance domains $A \in \sA(r)$ are contained in the face-cuts $\partial_H + (\Delta_H(r))_<$. Since each $\partial_H \subseteq H$ we have that $\partial A \subseteq \bigcup_{H \in \sH} H + \Delta_H(r)$. Hence each connected component of $\sD$ is contained in some unique acceptance domain, so $p(r) = \# \sA(r) \leq \# \sD$, as desired. \end{proof}

To obtain the upper bound of Theorem \ref{thm: generalised complexity} we wish to bound the number of face-cuts for each face.

\begin{lemma} \label{lem: face-cuts} For each $H \in \sH$ there is bounded set $U \subseteq \intl$ for which $\gamma \in \Delta_H(r)$ implies that $\gamma_< \in U$. \end{lemma}

\begin{proof} Choose some `centre' $c \in \partial_H$ and $\kappa$ for which $\partial_H \subseteq B_\kappa(c)$. If $\gamma \in \Delta_H(r)$ (i.e., $W \cap (\partial_H + \gamma_<)$ is non-trivial) then we claim that $\gamma_<$ belongs to a $\kappa$-neighbourhood of $W-c$. Indeed, let $w \in W \cap (\partial_H + \gamma_<)$. So $w-c \in W-c$ and $w - c\in \partial_H + \gamma_< - c \subseteq B_\kappa + \gamma_<$, that is, $\gamma_< \in B_\kappa(w-c)$. \end{proof}

\begin{lemma} \label{lem: number of face-cuts} Let $H \in \sH$. The number of distinct translates of $H$ by the elements of $(\Delta_H(r))_<$ is $\ll r^{\alpha_H'}$, where $\alpha_H' \coloneqq d - \rk(H) + \beta_H$. \end{lemma}

\begin{proof} Let $\Gamma^b$ be the rank $(k-d)$ subgroup of $\Gamma$ generated by the elements $b$ of Equation \ref{eq: H-basis}.

Consider the group $G = \Gamma /(\Gamma^H + \Gamma^b)$. Alternatively, it may be written as $\Gamma / (\Gamma^H + \Gamma^f)$, where $\Gamma^f$ is the rank $(k-d)-\beta_H$ free Abelian group generated by the vectors $f_i^H$ of Equation \ref{eq: H-basis}. We chose these basis vectors so that $\Gamma^H \cap \Gamma^f = \{0\}$, so it follows that $\Gamma^H + \Gamma^b$ has rank $\rk(H) + ((k-d)-\beta_H)$. Hence the rank of $G$ is $k$ minus this quantity, which is $\alpha_H'$.

The group $G$ is isomorphic to $\Z^{\alpha_H'} + T$, where $T$ is some finite torsion group. Let $t_1, \ldots, t_\ell \in \Gamma$ be a list of representatives for $T$. We let $G' \leqslant \Gamma$ be a free Abelian subgroup projecting isomorphically to the free part of $G$ under the quotient. So every element of $\gamma \in \Gamma$ may be written as
\[
\gamma = \gamma' + t_i + (\gamma^H + \gamma^b)
\]
where $\gamma' \in G'$, $\gamma^H \in \Gamma^H$ and $\gamma^b \in \Gamma^b$. Moreover, the choice of $\gamma'$ and $t_i$ is uniquely determined by $\gamma$.

We claim that for each $\gamma' \in G'$ and $\gamma^H \in \Gamma^H$, there are are most a bounded number of $t_i$ and $\gamma^b \in \Gamma^b$ for which $\gamma = (\gamma' + t_i + \gamma^b + \gamma^H) \in \Delta_H(r)$. There are only finitely many $t_i$, so there is no issue there. And since $b$ was chosen so that $b_<$ is a basis for $\intl$, there are only a bounded number of choices of $\gamma^b$ for which $(\gamma' + t_i + \gamma^H) + \gamma^b$ belongs to any fixed, bounded set $U$. Letting $U$ be as in Lemma \ref{lem: face-cuts}, we see that there are $\ll r^{\alpha_H'}$ elements $(\gamma' + t_i + \gamma^b + \gamma^H) \in \cB(r) \cap \Delta_H(r)$ for fixed $\gamma^H$. Adding an element of $\Gamma^H$ does not give a new translate of the hyperplane, so the result follows. \end{proof}

\begin{proof}[Proof that $p(r) \ll r^{\alpha'}$]
The proof is analogous to the proof that $p(r) \ll r^\alpha$ in the quasicanonical case. By Lemma \ref{lem: upper bound} we just need to bound the number of connected components coming from extended cuts $H + (\Delta_H(r))_<$. By Lemma \ref{lem: number of face-cuts} we can find the number of cuts in each direction and using Lemma \ref{lem: count components} we find that the number of connected components is $\ll r^{\alpha'_f}$ for each flag.
\end{proof}

\subsection{Proof of lower bound of Theorem \ref{thm: generalised complexity}}
To get a lower bound for $p(r)$ without the quasicanonical condition we use a slightly more delicate argument to construct `enough' acceptance domains of $\sA(r)$. The strategy will go as follows:
\begin{enumerate}
	\item for each flag $f \in \sF$, construct a small `box' $X$ with sides aligned with those of $f$, with $X$ given as an intersection of $\Gamma_<$ translates of $W$ or $W^c$;
	\item for each face of the window corresponding to each hyperplane $H$ of flag, define a collection $\Delta_H^N(r) \subseteq \cB_<(r)$ giving `nice cuts' of $X$ by the face $\partial_H$;
	\item show that the number of `nice cuts' with distinct intersection with $X$ in each direction is $\gg r^{\alpha'_H} \coloneqq r^{d - \rk(H) + \beta_H}$ and use these to construct $\gg r^{\alpha'}$ acceptance domains of $\sA(r)$.
\end{enumerate}

First we realise step 1:

\begin{lemma} \label{lem: box} Take a flag $f \in \sF$. Then there is a box $X \subseteq W$, of arbitrarily small given diameter, which is an intersection of $\Gamma_<$ translates of $W$ and $W^c$ whose sides are aligned with those of $f$, that is, it is an intersection of half-spaces whose associated hyperplanes are parallel to those of $f$. \end{lemma}

\begin{proof} By density of $\Gamma_<$, for each $H$ one can find a small element $\gamma_H \in \Gamma_<$ which lies on either side of $V(H)$. In particular, the $\gamma_H$ can be chosen so that the intersection $W \cap (W^c  + \gamma_H)$ results in a small `slab' $S_H$ of the form $H + I$, where $I$ is an interval orthogonal to $H$, at least considering $S_H$ near the centre of the face of $W$ parallel to $H$. One can repeat to create such a slab $S_H$ for each $H \in \sH$. Again by density of $\Gamma_<$ we may find lattice translates of the slabs with non-empty intersection, and thus of the form desired. \end{proof}

We now define what we mean by a `nice cut' in steps 2 and 3. Henceforth we fix some flag $f$ and small box $X$ aligned with $f$ as in the above lemma.

\begin{definition} Let $H\in f$. We say that $\partial_H + z$ is a {\bf nice cut} for $X$ if $(\partial_H + z) \cap X \neq \emptyset$ and $(W + z) \cap X = (H^+ + z) \cap X$. We let $\Delta_H^N(r)$ denote the set of $\gamma \in \cB(r)$ for which $\partial_H + \gamma_<$ is a nice cut. \end{definition}

\begin{figure}
	\def\svgwidth{0.5\textwidth}
	\centering
	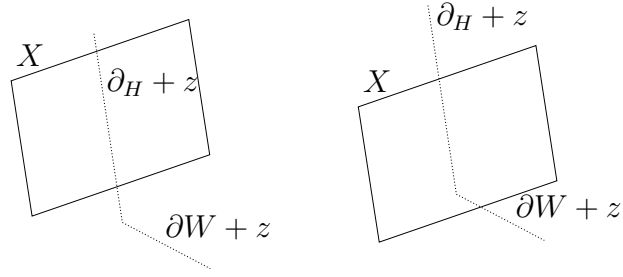
	\caption{The translate $\partial_H+z$ is a nice cut for the box $X$ if it cuts all the way through $X$. The cut on the left is nice, the cut on the right is not. }\label{fig:nice cuts}
\end{figure}

In other words, the translated face fully cuts the box, so without any $r$-dimensional facet of $\partial_H$ intersecting $X$ with $r < k-d-1$, and without interference from other defining hyperplanes of $W$ (see Figure \ref{fig:nice cuts}). To motivate the notation, notice that $\Delta_H^N(r)\subseteq \Delta_H(r)$. 

\begin{lemma} \label{lem: nice U}Let $H \in f$. If $X$ is sufficiently small then there exists some open set $U \in \intl$ for which any $z \in U$ gives a nice cut $\partial_H + z$. \end{lemma}

\begin{proof} Let $c$ be a `centre' of the face, a point of the interior of $\partial_H$. Take some $\kappa$, $\epsilon > 0$ with $W - c \supseteq B_{\kappa+\epsilon} \cap (H^+-c)$; since $c$ belongs to the interior of the face we can find a sufficiently small ball satisfying this. Let $x$ belong to the interior of $X$, which we assume is taken small enough to have diameter less than $\kappa$. Reduce $\epsilon$, if necessary, to ensure that $B_\epsilon(x) \subseteq X$.

We claim that if $z \in B_\epsilon(x - c)$ then $\partial_H + z$ is a nice cut. Firstly, we have that $(\partial_H + z) \cap X \neq \emptyset$. Indeed it contains $c + z$, since $d(x,c+z) = d(x - c,z) < \epsilon$, so $c + z \in X$. Secondly it cuts fully: suppose that $h \in H^+$ and $h + z \in X$. The latter implies that $h - c \in X - x - ( z - (x-c))$. By assumption $X - x \subseteq B_\kappa$ and $z - (x - c) \in B_\epsilon$, so $h - c \subseteq B_{\kappa + \epsilon}$. Since $h \in H^+$ too we have that $h \in W$, as desired. \end{proof}

So we can obtain a lower bound for the number of nice cuts $\partial_H + \gamma_<$, for $\gamma \in \Delta_H^N(r)$, by restricting to $U \cap \cB_<(r)$ for $U$ some open subset of $\intl$. By further restricting to those $\gamma$ resulting in distinct intersections $X \cap (H + \gamma_<)$, we complete step 3:

\begin{lemma} \label{lem: nice cuts} Let $H\in f$. There are $\gg r^{\alpha_H'}$ elements $\gamma \in \Delta_H^N(r)$ resulting in distinct intersections $X \cap (H + \gamma_<)$. \end{lemma}

\begin{proof}
For a subgroup $K$ of a free Abelian group $G$, let $K_\perp$ denote the group complementary to $K$ in $G$, i.e., $K \cap K_\perp = \{0\}$ and $K + K_\perp$ is finite index in $G$.

Let $U$ be as in Lemma \ref{lem: nice U}. Recall the number $\beta:=\beta_H$ from Equation \ref{eq: H-basis}. Choose any $\beta$ elements $h=(h_1, \dots, h_\beta)$ in $\Gamma^H$ which are linearly independent in the internal space $\intl$. Call the free group of rank $\beta$ that they span $A$. Take $A_\perp$ complementary to $A$ in $\Gamma$ and pick $(k - d - \beta)$ elements $f=(f_{\beta+1}, \dots, f_{k-d})$ from $A_\perp$ so that together with $h_<$ the projections $f_<$ form a basis for $\intl$. Note that since $A + A_\perp$ is rank $k$, it is finite index in $\Gamma$, so $A_< + (A_\perp)_<$ is dense in $\intl$. Therefore we can choose $f$ so that the $\Z$-span of $b = (h, f) \coloneqq (h_1,\ldots,h_\beta,f_{\beta+1},\ldots,f_{k-d})$ is as dense in $\intl$ as we wish. In particular, we may choose $f$ so that for all $x \in \intl$, there is $\ell\in \Z^{k-d}$ with $x+\ell \cdot b_<\in U$. Here, $\ell \cdot b_<=\ell_1 (b_1)_<+\dots +\ell_{k-d}(b_{k-d})_<$. Denote the $\Z$-span of $f$ by $B$, so that $A + B$ is rank $k-d$ and has as list of generators $b=(h, f)$ as in Equation \ref{eq: H-basis}.

Take $T\leqslant \Gamma$ such that $T$ is complementary to $\Gamma^H+B$. Then, since $\Gamma^H\cap B=\{0\}$, we have $\rk(T)=k-\rk(H) - (k-d-\beta)=\alpha_H'$, so that $\# (T\cap \cB(r))\gg r^{\alpha_H'}$. Further, by the choice of $b$, for every $t\in T\cap \cB(r)$ there is $\ell\in \Z^{k-d}$ such that $t_<+\ell \cdot b_<\in U$. By Lemma \ref{lem: nice U} this means that each $t+\ell \cdot b+\partial_H$ is a nice cut. Since $b_<$ is a basis for $\intl$, it is immediate that $\|t+\ell\cdot b\|\ll r$. 

To finish the proof we only need to check that each element of $T$ corresponds to a distinct nice cut. Indeed, let $t_1, t_2$ have $((t_1)_<+\ell_1\cdot b_< +\partial _H)\cap X=((t_2)_<+\ell_2\cdot b_< +\partial_H)\cap X$ for some $\ell_1, \ell_2\in \Z^{k-d}$. Since the cuts are nice, it follows that
\[
(t_1)_<+\ell_1\cdot b_<+H=(t_2)_<+\ell_2\cdot b_<+H,
\]
that is,
\[
(t_1 - t_2) + (\ell_1 - \ell_2)\cdot b \in \Gamma^H.
\]
Since $h_1, \dots, h_\beta\in \Gamma^H$, we can ignore them in the above equation. We thus obtain, for some $\ell_1, \ell_2\in \Z^{k-d-\beta}$,
\[
(t_1 - t_2) + (\ell_1 - \ell_2)\cdot f \in \Gamma^H.
\]
Since $t_1 - t_2\in T$, $(\ell_1 - \ell_2)\cdot f \in B$ and $T$ is complementary to $\Gamma^H+B$, the only way in which this can be is if both $t_1=t_2$ and $\ell_1=\ell_2$.

\end{proof}

\begin{proof}[Proof that $p(r) \gg r^{\alpha'}$]
Take a flag $f$ and construct a small box $X$ according to Lemma \ref{lem: box}, with $X$ an intersection of $\cB_<(c)$ translates $W$ and $W^c$. By Lemma \ref{lem: nice cuts} we can find $\gg r^{\alpha'_H}$ nice cuts of $X$ by faces in each direction $H$. The translates of these cuts can be used to form cut regions of $X$, by removing translated hyperplanes from $X$. By the definition of a nice cut it is easy to see that such regions may also be expressed via intersections of $W$ and $W^c$ translated by the same vectors. Indeed, expressing that $x \in X$ is above $H + z$ and below $H + z'$ is equivalent to saying that $x \in W + z$ and $x \in W^c + z'$ for a nice cut.

We see that the nice cuts define $\gg r^{\alpha'}$ regions which may be expressed as intersections of $\cB_<(r+c)$ translates of $W$ and $W^c$. Each acceptance domain of $\sA(r)$, by Equation \ref{eq: acceptance domain}, is contained in such a region, so there are at least this many patches and $p(r) \gg r^{\alpha'}$, as desired. \end{proof}

\section{The topology of the transversal and algebra of $\mathcal{C}$-topes} \label{sec: topology of transversal}

The results above show that neither the quasi- nor almost canonical conditions are strictly necessary for analysing the complexity of polytopal cut and project sets. However, we shall show in this section that the typically assumed almost canonical condition is not sufficient for constructions used in calculating topological invariants of polytopal cut and project sets, whereas the quasicanonical condition is sufficient.

\subsection{Background}

We briefly review how one associates topological invariants to aperiodic patterns, and how the structure of the internal space, projected lattice and window can in principle be used to determine these invariants for cut and project patterns. See Sadun's book \cite{Sad08book} for an introduction to the topological study of aperiodic patterns, and \cite{FHK02, GHK13} for the specific case of cut and project sets.

Given a tiling or Delone set (a relatively dense and uniformly discrete point set) $P$, there is an associated topological space $\Omega$, sometimes called the \emph{hull} of $P$. It has as points all patterns locally indistinguishable from $P$, that is, comprising of the same finite patches up to translation, at least in the case that $P$ has finite local complexity (FLC), as for our cut and project sets. It has a topology under which, loosely speaking, two patterns $P_1$ and $P_2$ are considered to be close if there are small perturbations (small translations, in the FLC case) of the two patterns which make them agree on a large ball about the origin. We consider the hull $\Omega$ of an aperiodic point pattern coming from a cut and project scheme with polytopal window. It has a {\bf (canonical) transversal} $\Xi \subseteq \Omega$, a subspace homeomorphic to a Cantor set consisting of those point patterns with a point lying over the origin.

There is a `cut version' $\intl^\Xi$ of the internal space. Restricted to a closed set corresponding to $W \subseteq \intl$ it is naturally identified with $\Xi$. There is a canonical action of $\Gamma$ on $\intl^\Xi$ and one may show that
\begin{equation} \label{eq: tiling space}
\Omega \cong (\phy \times \intl^\Xi)/\Gamma,
\end{equation}
where $\Gamma$ acts with the diagonal action on the product, and in the obvious way on $\phy$ (that is, by projecting to the physical space and then translating). Moreover, there is a map $\intl^\Xi \to \intl$, inducing a quotient map $\Omega \to (\phy \times \intl)/\Gamma \cong \tot/\Gamma \cong \bT^k$, which is the map to the \emph{maximal equicontinuous factor}. It commutes with the action of translation by $\phy$ and, for $\partial W$ with measure $0$, is almost everywhere one-to-one with respect to a uniquely defined probability measure, which establishes that the dynamical diffraction is pure point; see \cite{BLM07,LMS02}.

Equation \ref{eq: tiling space} is an important first step in the study of topological invariants of $\Omega$. Using that $\phy \times \intl^\Xi$ is homotopy equivalent to $\intl^\Xi$ and the Serre spectral sequence one finds that
\[
\check{H}^n(\Omega) \cong H^n(\Gamma;C(\intl^{\Xi},\Z)).
\]
The left-hand term is the \v{C}ech cohomology of $\Omega$ and the right-hand term is the group cohomology of $\Gamma$ with coefficients in the $\Z \Gamma$ module of continuous functions from $\intl^\Xi$ to $\Z$. Machinery for calculating this group cohomology is set out in \cite{FHK02, GHK13}. In fact, as pointed out in \cite[Appendix]{GHK13}, the approach is essentially constructs a transversal $X$ for the tiling space and a translation action on it, describing $\Omega$ as a \emph{suspension}
\[
\Omega \cong (\R^d \times X)/\Z^d
\]
for a Cantor set $X$ equipped with a $\Z^d$ action. Sadun and Williams \cite{SW03} showed that one may always write $\Omega$ as such a suspension, but this is of little computational help unless one has a handle on $X$ and the $\Z^d$ action on it. In the case of certain polytopal cut and project sets, one \emph{can} effectively describe such a $\Z^d$ action on $X$ by using the $\Gamma$ action on $\intl^\Xi$.

\subsection{The cut internal space}

Consider the set of singular points $S$ in the internal space, given by those points which lie in a lattice translate of the boundary of the window:
\[
S = \bigcup_{\gamma \in \Gamma} \partial W - \gamma_<.
\]
We let $NS$ denote its complement in $\intl$, the set of {\bf non-singular points}.

We define a new metric on $NS$ as follows. Given points $x$, $y \in NS$ and $\gamma \in \Gamma$, we write $x \Delta_\gamma y$ if $W-\gamma_<$ contains precisely one of $x$ or $y$; that is, $\partial W - \gamma_<$ separates $x$ and $y$. Then for distinct $x$, $y \in NS$ we let
\[
\delta(x,y) \coloneqq \sup \{1/r \mid x \Delta_\gamma y \text{ for } \gamma \in \cB(r)\}.
\]
Setting $\delta(x,x) = 0$, it is clear that $\delta(x,y) = 0$ if and only if $x = y$, by density of $\Gamma_<$. It is clearly symmetric. And if $x \Delta_\gamma z$ then either $x \Delta_\gamma y$ or $y \Delta_\gamma z$, since $y$ must belong either to $W - \gamma_<$ or its exterior. So the triangle inequality is satisfied: we have the ultrametric property $\delta(x,z) \leq \min\{\delta(x,y),\delta(y,z)\}$. We then take as our metric on $NS$ the maximum of the Euclidean metric and $\delta$:
\[
d(x,y) \coloneqq \max\{\|x-y\|,\delta(x,y)\}.
\]
We denote the completion of $NS$ under $d$ by $\intl^\Xi$. By the definition of the acceptance domains (see Equation \ref{eq: acceptance domain}), two points $x$, $y \in W \cap NS$ are close if and only if their corresponding cut and project sets $\Lambda_x$ and $\Lambda_y$ (see Equation \ref{eq: cpsfull}) agree to a large radius about the origin. This agrees with the metric on $\Omega$; we see that $W \cap NS$ corresponds to the non-singular patterns of the transversal $\Xi$ and its closure in $\intl^\Xi$ corresponds to the full canonical transversal $\Xi$.

The lattice $\Gamma$ acts on $NS$ by $\gamma \cdot x \mapsto x + \gamma_<$. Any given $\gamma$ acts uniformly continuously on $NS$, so this extends to a continuous group action of $\Gamma$ on $\intl^\Xi$. Since $\|x-y\| \leq d(x,y)$, the identity map on $NS$ extends to a continuous map from $\intl^\Xi$ onto $\intl$.

It is then a fact that $\Omega \cong (\phy \times \intl^\Xi)/\Gamma$, where $\Gamma$ acts diagonally on the product. We briefly sketch the reason (for more details, see \cite{FHK02}). To a point $x \in \intl^\Xi$ one may continuously assign a tiling $T(x) \in \Omega$. For points of $NS \subseteq \intl^\Xi$ this is just the tiling $\Lambda_x$, otherwise $T(x)$ is a limit of non-singular tilings. Every other tiling in $\Omega$ is a translate of such a tiling, so we have a surjective continuous map $f \colon \phy \times \intl^\Xi \to \Omega$ given by $f(v,x) \coloneqq T(x) + v$. It is not hard to see that $T(x)$ and $T(y)$ can only be translates of each other if $x$ and $y$ are in the same $\Gamma$-orbit, so by aperiodicity $T(x) + v$ and $T(y)+w$ agree if and only if $(v,x) \equiv (w,y) \mod \Gamma$.

\subsection{Hyperplane cuts and $\mathcal{C}$-topes}
We consider an alternative way of defining $\intl^\Xi$ which corresponds to a standard approach in the literature. For $x$, $y \in NS$, $\gamma \in \Gamma$ and $H \in \sH$, we let $x |^H_\gamma y$ if $x$ and $y$ lie on opposite sides of the translated hyperplane $H - \gamma_<$. Analogously to above, this defines a metric on $NS$:
\[
\delta'(x,y) \coloneqq \sup\{1/r \mid x |^H_\gamma y \text{ for } \gamma \in \cB(r), H \in \sH\}.
\]
In other words, $x$ and $y$ are close when separating them by a lattice translate of a hyperplane defining $\partial W$ requires a large lattice element. This defines an ultra-metric on $NS$, just as before. We then take our metric on $NS$ as
\[
d'(x,y) \coloneqq \max\{\|x-y\|,\delta'(x,y)\}.
\]
The completion of $(NS,d')$ is denoted by $\intl'$.

\begin{theorem} If the cut and project scheme is quasicanonical then a sequence is Cauchy in $(NS,d)$ if and only if it is Cauchy in $(NS,d')$. Hence, in this case, the identity map $\id \colon (NS,d) \to (NS,d')$ extends to a homeomorphism $\intl^\Xi \cong \intl'$ of the completions. \end{theorem}

\begin{proof} This follows from Proposition \ref{prop: halfspace covering} in an analogous way to the proof of Proposition \ref{prop: acc <-> cuts}. It may also be proven from Corollary \ref{cor: acceptance <-> cuts}, the main conclusion of Proposition \ref{prop: acc <-> cuts}. Firstly, note that if $x \Delta_\gamma y$ then $x |_\gamma^H y$ for some $H \in \sH$, so that $d(x,y) \leq d'(x,y)$. Hence, a sequence which is Cauchy with respect to $d'$ must also be Cauchy with respect to $d$. Conversely, Corollary \ref{cor: acceptance <-> cuts} implies that for sufficiently large $r$ and $x$, $y \in W \cap NS$, if $x \Delta_\gamma y$ does not hold for any $\gamma \in \cB(\kappa r + c)$ then $x |^H_{\gamma'} y$ does not hold for any $H \in \sH$ and $\gamma' \in \cB(r)$. So, for sufficiently large $r$, if $\delta(x,y) \leq (\kappa r+c)^{-1}$ then $\delta'(x,y) \leq r^{-1}$ for $x$, $y \in W \cap NS$. Hence, a sequence in $W \cap NS$ which is Cauchy with respect to $d$ is also Cauchy with respect to $d'$. Given a generic sequence in NS, not necessarily in $W$, which is Cauchy with respect to $d$, it must converge in $\intl$ (since $d(x,y) \leq \|x-y\|$). So we may translate it by an element of $\gamma_< \in \Gamma_<$ to a sequence which is eventually contained in $W \cap NS$, by density of $\Gamma_<$. Clearly translating by $\gamma_<$ does not change whether or not a sequence is Cauchy with respect to $d$ or $d'$, since points are separated by a lattice translate of the window or a hyperplane if and only if the points shifted by $\gamma_<$ are separated by the $\gamma_<$ shift of the window or hyperplane, which only changes the norm of the required translate by at most a constant, the norm of $\gamma$. We conclude that a sequence is Cauchy with respect to $d$ if and only if it is with respect to $d'$. \end{proof}

One may view the space $\intl'$ as a version of the internal space which is `cut' by lattice translates of the hyperplanes of $\sH$, and the above says that we may use this model to take the role of $\intl^\Xi$ in $\Omega \cong (\phy \times \intl^\Xi)/\Gamma$. The topology of the totally disconnected space $\intl'$ may be described in terms of the algebra of continuous functions $\intl' \to \Z$. The compactly supported such functions correspond precisely to sums of indicator functions of \emph{$\mathcal{C}$-topes}, polyhedral regions of $\intl$ whose boundaries are contained in the lattice translates of the hyperplanes, see \cite{FHK02, GHK13}.

The example below demonstrates that the almost canonical condition does not guarantee that the space $\intl'$ has the correct topology if we do not also assume the quasicanonical condition:

\begin{example} Consider the $3$-to-$1$ cut and project scheme of Example \ref{ex: triangle} with triangular window. As noted there, this scheme is almost canonical but not quasicanonical. Consider the three grey parallelogram regions in Figure \ref{fig: triangle}. These are $\mathcal{C}$-topes. Assigning different integers to each region describes a continuous function $f \colon \intl' \to \Z$. However, this function does \emph{not} correspond to a continuous function on the canonical transversal $\Xi \subseteq \Omega$. As proven in Example \ref{ex: triangle}, for arbitrarily large $r$ each grey region intersects a common acceptance domain of $r$-patch. So there are tilings of $\Xi$ which agree to arbitrarily large radii about the origin but are assigned different values under $f$, making $f$ discontinuous on $\Xi$. \end{example}

\bibliographystyle{amsalpha}
\bibliography{biblio}

\end{document}

%% file: images/cps.pdf_tex
%% Accompanies image file 'cps.pdf' (pdf, eps, ps)
\begingroup%
  \makeatletter%
  \providecommand\color[2][]{%
    \errmessage{(Inkscape) Color is used for the text in Inkscape, but the package 'color.sty' is not loaded}%
    \renewcommand\color[2][]{}%
  }%
  \providecommand\transparent[1]{%
    \errmessage{(Inkscape) Transparency is used (non-zero) for the text in Inkscape, but the package 'transparent.sty' is not loaded}%
    \renewcommand\transparent[1]{}%
  }%
  \providecommand\rotatebox[2]{#2}%
  \newcommand*\fsize{\dimexpr\f@size pt\relax}%
  \newcommand*\lineheight[1]{\fontsize{\fsize}{#1\fsize}\selectfont}%
  \ifx\svgwidth\undefined%
    \setlength{\unitlength}{234.70637512bp}%
    \ifx\svgscale\undefined%
      \relax%
    \else%
      \setlength{\unitlength}{\unitlength * \real{\svgscale}}%
    \fi%
  \else%
    \setlength{\unitlength}{\svgwidth}%
  \fi%
  \global\let\svgwidth\undefined%
  \global\let\svgscale\undefined%
  \makeatother%
  \begin{picture}(1,0.50363758)%
    \lineheight{1}%
    \setlength\tabcolsep{0pt}%
    \put(0,0){\includegraphics[width=\unitlength,page=1]{cps.pdf}}%
    \put(0.24419397,0.45868964){\color[rgb]{0,0,0}\makebox(0,0)[lt]{\lineheight{0}\smash{\begin{tabular}[t]{l}$\E_<$\end{tabular}}}}%
    \put(0.90131659,0.0570201){\color[rgb]{0,0,0}\makebox(0,0)[lt]{\lineheight{0}\smash{\begin{tabular}[t]{l}$\E_\vee$\end{tabular}}}}%
    \put(0.52331651,0.42991753){\color[rgb]{0,0,0}\makebox(0,0)[lt]{\lineheight{0}\smash{\begin{tabular}[t]{l}$\Gamma+s$\end{tabular}}}}%
    \put(0.54505755,0.06169416){\color[rgb]{0,0,0}\makebox(0,0)[lt]{\lineheight{0}\smash{\begin{tabular}[t]{l}$\cps_s$\end{tabular}}}}%
    \put(0.21920279,0.33103429){\color[rgb]{0,0,0}\makebox(0,0)[lt]{\lineheight{0}\smash{\begin{tabular}[t]{l}$W$\end{tabular}}}}%
    \put(0,0){\includegraphics[width=\unitlength,page=2]{cps.pdf}}%
    \put(0.93330906,0.4586658){\color[rgb]{0,0,0}\makebox(0,0)[lt]{\lineheight{0}\smash{\begin{tabular}[t]{l}$\E$\end{tabular}}}}%
  \end{picture}%
\endgroup%

%% file: images/q_can.pdf_tex
%% Creator: Inkscape 0.91_64bit, www.inkscape.org
%% Accompanies image file 'q_can.pdf' (pdf, eps, ps)
\begingroup%
  \makeatletter%
  \providecommand\color[2][]{%
    \errmessage{(Inkscape) Color is used for the text in Inkscape, but the package 'color.sty' is not loaded}%
    \renewcommand\color[2][]{}%
  }%
  \providecommand\transparent[1]{%
    \errmessage{(Inkscape) Transparency is used (non-zero) for the text in Inkscape, but the package 'transparent.sty' is not loaded}%
    \renewcommand\transparent[1]{}%
  }%
  \providecommand\rotatebox[2]{#2}%
  \ifx\svgwidth\undefined%
    \setlength{\unitlength}{595.27559055bp}%
    \ifx\svgscale\undefined%
      \relax%
    \else%
      \setlength{\unitlength}{\unitlength * \real{\svgscale}}%
    \fi%
  \else%
    \setlength{\unitlength}{\svgwidth}%
  \fi%
  \global\let\svgwidth\undefined%
  \global\let\svgscale\undefined%
  \makeatother%
  \begin{picture}(1,0.47619048)%
    \put(0,0){\includegraphics[width=\unitlength,page=1]{q_can.pdf}}%
  \end{picture}%
\endgroup%

%% file: images/hoopla.pdf_tex
%% Creator: Inkscape 0.91_64bit, www.inkscape.org
%% Accompanies image file 'hoopla.pdf' (pdf, eps, ps)
\begingroup%
  \makeatletter%
  \providecommand\color[2][]{%
    \errmessage{(Inkscape) Color is used for the text in Inkscape, but the package 'color.sty' is not loaded}%
    \renewcommand\color[2][]{}%
  }%
  \providecommand\transparent[1]{%
    \errmessage{(Inkscape) Transparency is used (non-zero) for the text in Inkscape, but the package 'transparent.sty' is not loaded}%
    \renewcommand\transparent[1]{}%
  }%
  \providecommand\rotatebox[2]{#2}%
  \ifx\svgwidth\undefined%
    \setlength{\unitlength}{226.77165354bp}%
    \ifx\svgscale\undefined%
      \relax%
    \else%
      \setlength{\unitlength}{\unitlength * \real{\svgscale}}%
    \fi%
  \else%
    \setlength{\unitlength}{\svgwidth}%
  \fi%
  \global\let\svgwidth\undefined%
  \global\let\svgscale\undefined%
  \makeatother%
  \begin{picture}(1,0.48361273)%
    \put(0,0){\includegraphics[width=\unitlength,page=1]{hoopla.pdf}}%
  \end{picture}%
\endgroup%

%% file: images/bad_ac.pdf_tex
%% Creator: Inkscape 0.91_64bit, www.inkscape.org
%% Accompanies image file 'bad_ac.pdf' (pdf, eps, ps)
\begingroup%
  \makeatletter%
  \providecommand\color[2][]{%
    \errmessage{(Inkscape) Color is used for the text in Inkscape, but the package 'color.sty' is not loaded}%
    \renewcommand\color[2][]{}%
  }%
  \providecommand\transparent[1]{%
    \errmessage{(Inkscape) Transparency is used (non-zero) for the text in Inkscape, but the package 'transparent.sty' is not loaded}%
    \renewcommand\transparent[1]{}%
  }%
  \providecommand\rotatebox[2]{#2}%
  \ifx\svgwidth\undefined%
    \setlength{\unitlength}{153.09417965bp}%
    \ifx\svgscale\undefined%
      \relax%
    \else%
      \setlength{\unitlength}{\unitlength * \real{\svgscale}}%
    \fi%
  \else%
    \setlength{\unitlength}{\svgwidth}%
  \fi%
  \global\let\svgwidth\undefined%
  \global\let\svgscale\undefined%
  \makeatother%
  \begin{picture}(1,0.80760926)%
    \put(0,0){\includegraphics[width=\unitlength,page=1]{bad_ac.pdf}}%
  \end{picture}%
\endgroup%

%% file: images/nice_cuts.pdf_tex
%% Accompanies image file 'nice_cuts.pdf' (pdf, eps, ps)
\begingroup%
  \makeatletter%
  \providecommand\color[2][]{%
    \errmessage{(Inkscape) Color is used for the text in Inkscape, but the package 'color.sty' is not loaded}%
    \renewcommand\color[2][]{}%
  }%
  \providecommand\transparent[1]{%
    \errmessage{(Inkscape) Transparency is used (non-zero) for the text in Inkscape, but the package 'transparent.sty' is not loaded}%
    \renewcommand\transparent[1]{}%
  }%
  \providecommand\rotatebox[2]{#2}%
  \newcommand*\fsize{\dimexpr\f@size pt\relax}%
  \newcommand*\lineheight[1]{\fontsize{\fsize}{#1\fsize}\selectfont}%
  \ifx\svgwidth\undefined%
    \setlength{\unitlength}{835.94252699bp}%
    \ifx\svgscale\undefined%
      \relax%
    \else%
      \setlength{\unitlength}{\unitlength * \real{\svgscale}}%
    \fi%
  \else%
    \setlength{\unitlength}{\svgwidth}%
  \fi%
  \global\let\svgwidth\undefined%
  \global\let\svgscale\undefined%
  \makeatother%
  \begin{picture}(1,0.44652608)%
    \lineheight{1}%
    \setlength\tabcolsep{0pt}%
    \put(0,0){\includegraphics[width=\unitlength,page=1]{nice_cuts.pdf}}%
    \put(0.00899004,0.34332468){\color[rgb]{0,0,0}\makebox(0,0)[lt]{\lineheight{1.25}\smash{\begin{tabular}[t]{l}$X$\end{tabular}}}}%
    \put(0,0){\includegraphics[width=\unitlength,page=2]{nice_cuts.pdf}}%
    \put(0.26184762,0.0542151){\color[rgb]{0,0,0}\makebox(0,0)[lt]{\lineheight{1.25}\smash{\begin{tabular}[t]{l}$\partial W+z$\end{tabular}}}}%
    \put(0,0){\includegraphics[width=\unitlength,page=3]{nice_cuts.pdf}}%
    \put(0.16268076,0.30156819){\color[rgb]{0,0,0}\makebox(0,0)[lt]{\lineheight{1.25}\smash{\begin{tabular}[t]{l}$\partial_H+z$\end{tabular}}}}%
    \put(0,0){\includegraphics[width=\unitlength,page=4]{nice_cuts.pdf}}%
    \put(0.59702491,0.29918565){\color[rgb]{0,0,0}\makebox(0,0)[lt]{\lineheight{1.25}\smash{\begin{tabular}[t]{l}$X$\end{tabular}}}}%
    \put(0,0){\includegraphics[width=\unitlength,page=5]{nice_cuts.pdf}}%
    \put(0.7209547,0.41306289){\color[rgb]{0,0,0}\makebox(0,0)[lt]{\lineheight{1.25}\smash{\begin{tabular}[t]{l}$\partial _H+z$\end{tabular}}}}%
    \put(0,0){\includegraphics[width=\unitlength,page=6]{nice_cuts.pdf}}%
    \put(0.85566444,0.08869275){\color[rgb]{0,0,0}\makebox(0,0)[lt]{\lineheight{1.25}\smash{\begin{tabular}[t]{l}$\partial W+z$\end{tabular}}}}%
  \end{picture}%
\endgroup%

%% file: complexity_polytopal.bbl
\providecommand{\bysame}{\leavevmode\hbox to3em{\hrulefill}\thinspace}
\providecommand{\MR}{\relax\ifhmode\unskip\space\fi MR }
\providecommand{\MRhref}[2]{%
  \href{http://www.ams.org/mathscinet-getitem?mr=#1}{#2}
}
\providecommand{\href}[2]{#2}
\begin{thebibliography}{HKWS16}

\bibitem[BG09]{BG09}
Winfried Bruns and Joseph Gubeladze, \emph{Polytopes, rings, and {$K$}-theory},
  Springer Monographs in Mathematics, Springer, Dordrecht, 2009. \MR{2508056}

\bibitem[BG13]{BG13}
Michael Baake and Uwe Grimm, \emph{Aperiodic order. {V}ol. 1}, Encyclopedia of
  Mathematics and its Applications, vol. 149, Cambridge University Press,
  Cambridge, 2013, A mathematical invitation, With a foreword by Roger Penrose.
  \MR{3136260}

\bibitem[BLM07]{BLM07}
Michael Baake, Daniel Lenz, and Robert~V. Moody, \emph{Characterization of
  model sets by dynamical systems}, Ergodic Theory Dynam. Systems \textbf{27}
  (2007), no.~2, 341--382. \MR{2308136}

\bibitem[BV00]{BV00}
Val{\'e}rie Berth{\'e} and Laurent Vuillon, \emph{Tilings and rotations on the
  torus: a two-dimensional generalization of {S}turmian sequences}, Discrete
  Math. \textbf{223} (2000), no.~1-3, 27--53. \MR{1782038 (2001i:68139)}

\bibitem[FHK02]{FHK02}
Alan Forrest, John Hunton, and Johannes Kellendonk, \emph{Topological
  invariants for projection method patterns}, Mem. Amer. Math. Soc.
  \textbf{159} (2002), no.~758, x+120. \MR{1922206 (2003j:37024)}

\bibitem[GHK13]{GHK13}
Franz G{\"a}hler, John Hunton, and Johannes Kellendonk, \emph{Integral
  cohomology of rational projection method patterns}, Algebr. Geom. Topol.
  \textbf{13} (2013), no.~3, 1661--1708. \MR{3071138}

\bibitem[HKW14]{HKW14}
Alan Haynes, Michael Kelly, and Barak Weiss, \emph{Equivalence relations on
  separated nets arising from linear toral flows}, Proc. Lond. Math. Soc. (3)
  \textbf{109} (2014), no.~5, 1203--1228. \MR{3283615}

\bibitem[HKW18]{HaynKoivWalt2015a}
Alan Haynes, Henna Koivusalo, and James Walton, \emph{A characterization of
  linearly repetitive cut and project sets}, Nonlinearity \textbf{31} (2018),
  no.~2, 515--539. \MR{3755878}

\bibitem[HKWS16]{HaynKoivSaduWalt2015}
Alan Haynes, Henna Koivusalo, James Walton, and Lorenzo Sadun, \emph{Gaps
  problems and frequencies of patches in cut and project sets}, Math. Proc.
  Cambridge Philos. Soc. \textbf{161} (2016), no.~1, 65--85. \MR{3505670}

\bibitem[Hof95a]{Hof95b}
A.~Hof, \emph{Diffraction by aperiodic structures at high temperatures}, J.
  Phys. A \textbf{28} (1995), no.~1, 57--62. \MR{1325836}

\bibitem[Hof95b]{Hof95a}
\bysame, \emph{On diffraction by aperiodic structures}, Comm. Math. Phys.
  \textbf{169} (1995), no.~1, 25--43. \MR{1328260}

\bibitem[Jul10]{Jul10}
Antoine Julien, \emph{Complexity and cohomology for cut-and-projection
  tilings}, Ergodic Theory Dynam. Systems \textbf{30} (2010), no.~2, 489--523.
  \MR{2599890 (2011i:52045)}

\bibitem[LMS02]{LMS02}
J.-Y. Lee, R.~V. Moody, and B.~Solomyak, \emph{Pure point dynamical and
  diffraction spectra}, Ann. Henri Poincar\'e \textbf{3} (2002), no.~5,
  1003--1018. \MR{1937612 (2004a:52040)}

\bibitem[LP03]{LP03}
Jeffrey~C. Lagarias and Peter A.~B. Pleasants, \emph{Repetitive {D}elone sets
  and quasicrystals}, Ergodic Theory Dynam. Systems \textbf{23} (2003), no.~3,
  831--867. \MR{1992666 (2005a:52018)}

\bibitem[Moo97]{Moo97}
Robert~V. Moody, \emph{Meyer sets and their duals}, The mathematics of
  long-range aperiodic order ({W}aterloo, {ON}, 1995), NATO Adv. Sci. Inst.
  Ser. C Math. Phys. Sci., vol. 489, Kluwer Acad. Publ., Dordrecht, 1997,
  pp.~403--441. \MR{1460032 (98e:52029)}

\bibitem[MS15]{MS15}
Jens Marklof and Andreas Str\"{o}mbergsson, \emph{Visibility and directions in
  quasicrystals}, Int. Math. Res. Not. IMRN (2015), no.~15, 6588--6617.
  \MR{3384490}

\bibitem[Sad08]{Sad08book}
Lorenzo Sadun, \emph{Topology of tiling spaces}, University Lecture Series,
  vol.~46, American Mathematical Society, Providence, RI, 2008. \MR{2446623
  (2009m:52041)}

\bibitem[SBGC84]{SchBleGraCah84}
D.~Shechtman, I.~Blech, D.~Gratias, and J.~W. Cahn, \emph{Metallic phase with
  long-range orientational order and no translational symmetry}, Phys. Rev.
  Lett. \textbf{53} (1984), 1951--1953.

\bibitem[SW03]{SW03}
Lorenzo Sadun and R.~F. Williams, \emph{Tiling spaces are {C}antor set fiber
  bundles}, Ergodic Theory Dynam. Systems \textbf{23} (2003), no.~1, 307--316.
  \MR{1971208 (2004a:37023)}

\end{thebibliography}
